\documentclass[final,a4paper]{siamltex}
  \usepackage{amsmath}
  \usepackage{amssymb}
 \usepackage{a4wide}
  \usepackage{paralist}
  \usepackage{graphics} %% add this and next lines if pictures should be in esp format
  \usepackage{epsfig} %For pictures: screened artwork should be set up with an 85 or 100 line screen
  \usepackage{subfigure}
    \usepackage{color}

\newtheorem{remark}{Remark}
\newtheorem{assumption}{Assumption}

\newcommand {\be}{\begin{equation}}
\newcommand {\ee}{\end{equation}}  
\newcommand {\eps} {\varepsilon}
\newcommand {\alp} {\alpha}
\newcommand {\la} {\lambda}

\newcommand {\deri}[2] {\frac {\partial #1}{\partial #2}}
\newcommand {\derd}[3] {\frac {\partial^{2} #1}{\partial #2 \partial 
#3}}

\title{On simultaneous identification of the shape and generalized impedance boundary condition in obstacle scattering}

\author{Laurent Bourgeois\thanks{Laboratoire POEMS, ENSTA, 
 32, Boulevard Victor, 75739 Paris Cedex 15, France} 
        \and Nicolas Chaulet 
 \thanks{CMAP, INRIA Saclay Ile de France and Ecole Polytechnique, Route de Saclay, 91128 Palaiseau Cedex, France}
    \and Houssem Haddar 
   \thanks{CMAP, INRIA Saclay Ile de France and Ecole Polytechnique,
   Route de Saclay, 91128 Palaiseau Cedex, France}}

\begin{document}
\maketitle

\begin{abstract} 
We consider the inverse obstacle scattering problem of determining both the shape and the
``equivalent impedance'' from far field measurements at
a fixed frequency. In this work, the surface impedance is represented
by a
second order surface differential operator (refer to as generalized impedance boundary condition) as opposed to a scalar function. The generalized impedance boundary condition can be seen as a more accurate
model for effective impedances and is widely used in the scattering problem for thin coatings. Our approach is based on a least square optimization technique. A major part of our analysis is to characterize the derivative of the cost function with respect to the boundary and this complex surface impedance configuration.
In particular, we provide an extension of the notion of shape
derivative to the case where the involved impedance parameters do not need to be surface traces of given
functions, which leads (in general) to a non-vanishing 
tangential boundary perturbation. 
The efficiency of considering this type of derivative 
is 
illustrated by several 2D numerical experiments based on a (classical) steepest
descent method. The feasibility of retrieving both the shape and the
impedance parameters is also discussed in our numerical experiments.
\end{abstract}



\pagestyle{myheadings}
\thispagestyle{plain}
\markboth{On simultaneous identification of the shape of an obstacle and its GIBC}{}

\section{Introduction}
The identification of complex targets from measurements of
scattered waves is an important problem in inverse scattering theory that arises
in many real life applications, such as in non-destructive testing, and radar or
sonar applications. The solution of underlying inverse problem is challenging due to the
inherent ill-posedness and also the non-linearity of the problem. Therefore, it is desirable to base the inversion algorithms on a simplified direct scattering model, whenever possible.
Such an example, is the modeling of the direct scattering problem for an imperfectly absorbing target or an obstacle coated by a thin layer, as boundary value problems with complicated surface impedance known as Generalized Impedance Boundary   
Conditions (GIBC). 
The most simplified model is the so-called impedance or Robin-Fourier boundary
condition. However, as demonstrated in many recent studies (see \cite{BenLem96,DurHadJol06,HadJol02,HadJolNgu05,haddar}), 
such impedance boundary condition can not accurately model many instances of complex materials in the case of  
multi-angle illuminations of
coated (and/or corrugated) surfaces. For such materials, a better approximation is the surface impedance in the form of a second order surface operator.
The goal of this work is to solve the inverse scattering problem for targets made of complex material based on an approximate model of the direct scattering as a boundary value problem with complex surface impedance. More precisely we shall consider the scalar problem for the Helmholtz equation
(which corresponds to the scattering of acoustic waves, or electromagnetic waves specially polarized)
and consider GIBC of the form
\[
\deri{u}{\nu} + {\rm div}_{\Gamma} (\mu \nabla_{\Gamma}u)+\la u=0 \quad \mbox{ on } \Gamma,
\]
where $\Gamma$ is the boundary of an obstacle $D$, $\mu$ and $\la$ are complex
valued functions, ${\rm div}_{\Gamma}$ and $\nabla_{\Gamma}$ 
the surface divergence and the surface gradient on $\Gamma$, respectively and $\nu$ denotes
the unit normal to $\Gamma$ directed to the exterior of $D$. 
As a particular example of electromagnetic scattering, if the scatterer is a perfect conductor coated with a thin
dielectric layer, for the transverse electric polarization, 
the  GIBC model corresponds to  $\mu=\delta$ and $\lambda=\delta k^2 n$, where
$k$ denotes the wave number, $\delta$ is the (possibly non constant) width of
the layer and $n$ is its refractive index \cite{haddar}.

The inverse problem under investigation is then to reconstruct both the obstacle
$D$ and the coefficients $\la$ and $\mu$ from far fields measurements at a fixed frequency (wave number).
This problem for the case of $\mu=0$ (which corresponds to the classical impedance boundary
condition) has been addressed for instance in \cite{CaCo04,Ser06,LiuNakSin07,HeKinSin09} for the Helmholtz equation and in \cite{Ru08,Ba09,Si10,CaKrSc10}  for the Laplace equation. The case $\mu \neq 0$ has been considered in
\cite{bourgeois_haddar,BoChHa11} where the main interest was prove the unique characterization of 
the impedance operator from a single measurement assuming that the boundary is known a priori, as well as the study of the question of stability, in particular when only an approximation of such boundary is known. In the present work, in addition to the impedance operator, the geometry of the obstacle is also unknown and hence we assume that measurements corresponding to several incident plane waves are available.

We  first prove the uniqueness for both $D$ and ($\la$,$\mu$) from a knowledge of the far field patterns for all possible incident and
observation directions.
Our proof
uses the technique of mixed reciprocity relation as in \cite{kirsch_kress,kress_rundell}. 
To solve the inverse problem we use an iterative Newton approach for nonlinear optimization problems.
The
main contribution of our analysis is the characterization of the  derivative of the far field with respect 
to the domain $D$ when the impedance parameters $(\la,\mu)$ are unknown. 
At a first glance, this may appear as a simple exercise using shape derivative
tools, as described for example in the monographs \cite{All,HenPie,sokzol}.
 However, it turned out that the following two specific issues of the problem seriously complicate this task. The first one is due to the 
 second-order surface differential operator appearing in the expression of the generalized impedance boundary condition.
 As it will become clear from the paper, this requires involved calculations which lead to non-intuitive final expressions for the derivatives. 
\\
The second major issue is not
attributed to GIBC but rather to the fact that the coefficients $\la$ and/or $\mu$
are unknown functions supported on the unknown boundary $\partial D$. This
configuration makes unclear the definition of the partial derivative with respect to $D$.
To overcome this ambiguity we shall adapt the usual definition of partial
derivative with respect to $D$ by defining an appropriate
extension of the impedance parameters $\la$ and $\mu$ to the perturbed boundary $\partial
D_\eps=\partial D + \eps(\partial D)$, where $\eps$ is a perturbation of
$\partial D$ (see definition \ref{frechet} hereafter).
As a surprising result, the obtained derivative depends not only on the normal component of the
perturbation $\eps$ but also on the tangential component (see Theorem \ref{main}
hereafter). In order to derive these results we adopted a method based on
integral representation of the solutions as introduced in \cite{KrePai} for the
case of Dirichlet or Neumann boundary conditions and in \cite{haddar_kress} for
the case of constant impedance boundary conditions. 
% More precisely, our paper can be viewed as a continuation of \cite{KrePai} for the Neumann boundary condition and of \cite{haddar_kress} for the impedance boundary condition with constant $\la$, 
% in the sense our paper is based on some integral representation of the
% scattered fields.
An alternative technique could be the use of shape derivative tools, but
this is not subject of our study as we believe that it is probably more technical.\\
The last part of our work is dedicated to the investigation of the numerical reconstruction algorithm based on a
classical least square optimization formulation of the problem which is solved by using a steepest descent method with regularization of the descent
direction. The
forward problem is solved by using a finite element method and the
reconstruction update is based on a boundary variation technique (which requires 
re-meshing of the computational domain at each step).   
The goal of our numerical study is first to illustrate the efficiency of
considering the proposed non conventional form of the shape derivative, and second 
to discuss the feasibility of retrieving both the obstacle and the
impedance functions in the impedance operator. 

The outline of the paper is as follows. In Section 2, we describe the direct and inverse problems. In Section 3, 
we prove uniqueness for the considered inverse problem, whereas Section 4 is
dedicated to the evaluation of the  derivative of the far field with respect to
the boundary of the obstacle. In Section 5 we describe the optimization technique based on a
least square formulation of the inverse problem. Various numerical tests in the $2D$ case
showing the efficiency of the proposed steepest descent method are presented in
Section 6. A technical lemma related to some differential geometry identities that are used in our analysis is presented in an
Appendix.

\section{The statement of the inverse problem} 
Let $D$ be an open bounded domain of $\mathbb{R}^d$, with $d=2$ or $3$, the boundary $\partial D$ of which is Lipschitz continuous, such that $\Omega=\mathbb{R}^d \setminus \overline{D}$ is connected and let $(\la,\mu)\in (L^{\infty}(\partial D))^2$ be some impedance coefficients. The scattering problem with generalized impedance boundary conditions (GIBC) consists in finding $u = u^s+ u^i$ such that
\begin{equation}
\label{PbInit}
\left\{
\begin{aligned}
& \Delta u + k^2 u = 0  \quad \text{in} \; \Omega\\
& \frac{\partial u}{\partial \nu} + {\rm div}_{\Gamma} (\mu \nabla_{\Gamma} u)+ \la u =0  \quad \text{on} \; \partial D\\
& \lim \limits_{R \to \infty} \int_{|x|=R}\left|\deri{u^s}{r} - iku^s \right|^2ds =0.
\end{aligned}
\right.
\end{equation}
Here
$k$ is the wave number, $u^i = e^{ik\hat{d}\cdot x}$ is an incident plane wave where $\hat{d} $ belongs to the unit sphere of $\mathbb{R}^d$ denoted $S^{d-1}$, and
$u^s \in V(\Omega) :=  \{ v \in \mathcal{D}'(\Omega) ,\varphi v \, \in \, H^1(\Omega) \, \forall \, \varphi \, \in \mathcal{D}(\mathbb{R}^d) \textrm{  and  }  v_{|\partial D} \, \in \,  H^1(\partial D) \}$ is the scattered field.\\
The surface operators ${\rm div}_{\Gamma}$ and $\nabla_{\Gamma}$ are precisely defined in Chapter 5 of \cite{HenPie}. For $v \in H^1(\partial D)$ the surface gradient $\nabla_{\Gamma}v$ lies in  $L^2_\Gamma(\partial D):=\{V \in L^2(\partial D,\mathbb{R}^d)\, ,\, V\cdot\nu=0\}$ while ${\rm div}_{\Gamma} (\mu \nabla_{\Gamma} u)$ is defined in $H^{-1}(\partial D)$ for $\mu \in L^{\infty}(\partial D)$ by
\begin{equation} 
 \langle {\rm div}_{\Gamma} (\mu \nabla_{\Gamma} u), v \rangle_{H^{-1}(\partial D),H^{1}(\partial D)} := - \int_{\partial D} \mu \nabla_{\Gamma}u \cdot\nabla_{\Gamma}v\, ds \quad \forall v \in H^{1}(\partial D).
\label{bypart}
\end{equation}
The last equation in (\ref{PbInit}) is the classical Sommerfeld radiation condition.
The proof for well--posedness of problem (\ref{PbInit}) and the numerical computation of its solution can be done using the so--called Dirichlet--to--Neumann map so that we can give an equivalent formulation of (\ref{PbInit}) in a bounded domain $\Omega_R = \Omega \cap B_R$ where $B_R$ is the ball of radius $R$ such that $D \subset B_R$. The Dirichlet--to--Neumann map, $S_R : H^{1/2}(\partial B_R) \mapsto  H^{-1/2}(\partial B_R) $ is defined for $g \in H^{1/2}(\partial B_R)$ by $S_Rg :={ \partial{u^e}/\partial{r}}|_{{\partial B_R}}$ where $u^e \in V(\mathbb{R}^d \setminus \overline{B_R})$ is the radiating solution of the Helmholtz equation outside $B_R$  and $u^e = g$ on $\partial B_R$.\\ 
Solving (\ref{PbInit}) is equivalent to find  $u$ in $V_R:=\{v \, \in \, H^1(\Omega_R);\, v_{|\partial D} \, \in \,  H^1(\partial D)\}$ such that:
\begin{equation}
\label{PbBorne}
\left\{
\begin{aligned}
&\Delta u + k^2u = 0  \quad \textrm{in} \; \Omega_R\\
& \frac{\partial u}{\partial \nu} + {\rm div}_{\Gamma} (\mu \nabla_{\Gamma} u) + \lambda u =0  \quad \textrm{on} \; \partial D\\
& \deri{u}{r} - S_R(u) =\deri{u^i}{r}-S_R(u^i) \; \text{ on } \partial B_R.
\end{aligned}
\right.
\end{equation}
We introduce the assumption
\begin{assumption}
\label{AsPbdirect}
The coefficients $(\lambda,\mu) \in (L^\infty(\partial D))^2$ are such that
\[
 \Im m(\lambda)\geq 0,\; \Im m(\mu)\leq 0 \quad \text{ a.e. in } \partial D
\]
and there exists $c>0$ such that 
\[
 \Re e(\mu) \geq c \qquad \text{ a.e. in } \partial D.
 \]
\end{assumption}
\noindent Well--posedness of problem (\ref{PbBorne}) is established in the following theorem, the proof of which is classical and given in \cite{BoChHa10report}.
\begin{theorem}
\label{ThDirect}
With assumption \ref{AsPbdirect} the problem (\ref{PbBorne}) has a unique solution $u$ in $V_R$.
\end{theorem}
\\
\noindent In order to define the inverse problem, we recall now the definition of the far field associated to a scattered field. 
From \cite{book-CK98}, the scattered field has the asymptotic behavior:
\[
 u^s(x) = \frac{e^{ikr}}{r^{(d-1)/2}}\left(u^{\infty}(\hat{x}) + \mathcal{O}\left(\frac{1}{r}\right)\right)\qquad r \longrightarrow +\infty
\]
uniformly for all the directions $\hat{x} = x/r \in S^{d-1}$ with $r=|x|$, and the far field $u^{\infty} \in L^2(S^{d-1})$ has the following integral representation
\begin{equation}
\label{uinf}
 u^{\infty}(\hat{x}) = \int_{\partial D}\left( u^s(y)\frac{\partial \Phi ^{\infty}(\hat{x},y)}{\partial\nu(y)} - \frac{\partial u^s(y)}{\partial\nu}\Phi ^{\infty}(\hat{x},y) \right) ds(y) \quad \forall \hat{x} \in S^{d-1}.
\end{equation}
Here $\Phi ^{\infty}(\cdot,y)$ is the far field associated with the Green function $\Phi(\cdot,y)$ of the Helmholtz equation.
The function $\Phi(\cdot,y)$ is defined in $\mathbb{R}^2$ by $\Phi(x,y)=(i/4)H_0^1(k|x-y|)$, where $H_0^1$ is the Hankel function of the first kind and of order $0$,
and in $\mathbb{R}^3$ by $e^{ik|x-y|}/(4\pi|x-y|)$. The associated far fields are defined in $S^1$ by $(e^{i\pi/4}/\sqrt{8\pi k})e^{-iky\cdot \hat{x}}$ and in $S^2$ by
$(1/4\pi) e^{-iky\cdot \hat{x}}$ respectively. The second integral in (\ref{uinf}) has to be understood as a duality pairing between $H^{-1/2}(\partial D)$ and  $H^{1/2}(\partial D)$.
\noindent We are now in a position to introduce the far field map $T$,
\[
 T : \quad (\lambda,\mu,\partial{D}) \rightarrow u^{\infty}
\]
where $u^{\infty}$ is the far field associated with the scattered field $u^s = u - u^i$ and $u$ is the unique solution of problem (\ref{PbInit}) with obstacle $D$ and impedances $(\la,\mu)$ on $\partial D$. \\
The general inverse problem we are interested in is the following: given incident plane waves of directions $\hat{d} \in S^{d-1}$, is it possible to reconstruct the obstacle $D$ as well as the impedances $\lambda$ and $\mu$ defined on $\partial D$ from the corresponding far fields $u^{\infty}=T(\lambda,\mu,\partial{D})$?
The first question of interest is the identifiability of $(\lambda,\mu,\partial{D})$ from the far field data.
\section{A uniqueness result}
In this section, we provide a uniqueness result concerning identification of both the obstacle $D$ and the impedances $(\lambda,\mu)$ from the far fields associated to plane waves with all incident directions $\hat{d} \in S^{d-1}$. In this respect we denote by $u^\infty(\hat{x},\hat{d})$ the far field in the $\hat{x}$ direction associated to the plane wave with direction $\hat{d}$.
In the following, we introduce some regularity assumptions for the obstacle $D$ and the impedances $\la,\mu$.
\begin{assumption}
\label{AsPbinverse}
The boundary $\partial D$ is $C^2$, and the impedances satisfy $\la \in C^0(\partial D)$ and $\mu \in C^1(\partial D)$. 
\end{assumption}
\\
\noindent The main result is the following theorem, which is a generalization of the uniqueness result for $\mu=0$ proved in \cite{kress_rundell}.
\begin{theorem}
\label{1}
Assume that $(\la_1,\mu_1,\partial D_1)$ and $(\la_2,\mu_2,\partial D_2)$ satisfy assumptions \ref{AsPbdirect} and \ref{AsPbinverse}, and the corresponding far fields $u_1^\infty=T(\la_1,\mu_1,\partial D_1)$ and $u_2^\infty=T(\la_2,\mu_2,\partial D_2)$
satisfy $u_1^\infty(\hat{x},\hat{d})=u_2^\infty(\hat{x},\hat{d})$ for all $\hat{x} \in S^{d-1}$ and $ \hat{d} \in S^{d-1}$. Then $D_1=D_2$ and $(\la_1,\mu_1)=(\la_2,\mu_2)$.
\end{theorem}
\\
\noindent The proof of the above theorem is based on several results, the first one
is the mixed reciprocity lemma and does not require the regularity assumption \ref{AsPbinverse}. 
\begin{lemma}
\label{reciprocity}
Let $w^\infty(\cdot,z)$ be the far field associated to the incident field $\Phi(\cdot,z)$ with $z \in \Omega$, and $u^s(\cdot,\hat{x})$ be the scattered field associated to the plane wave of direction $\hat{x} \in S^{d-1}$.
Then 
\[w^\infty(-\hat{x},z)=\gamma(d)\,u^s(z,\hat{x}),\]
with $\gamma(2)= e^{i\pi/4}/\sqrt{8\pi k}$ and  $\gamma(3)=1/4\pi$.
\end{lemma}
\begin{proof}
For two incident fields $u^i_1$ and $u^i_2$, the associated total fields $u_1$ and $u_2$ satisfy
\begin{align*}
\int_{\partial D}\left(u_1\deri{u_2}{\nu} \right. &\left.-u_2\deri{u_1}{\nu}\right)ds \\
= \int_{\partial D}\left(\mu \nabla_{\Gamma}u_1\cdot\nabla_{\Gamma}u_2-\right. &\left. \la u_1 u_2  - \mu \nabla_{\Gamma}u_1\cdot\nabla_{\Gamma}u_2+ \la u_1 u_2\right)ds = 0.
\end{align*}
By using the decomposition $u_1=u^i_1+u^s_1$ and $u_2=u^i_2+u^s_2$, that the
incident fields solve the Helmholtz equation inside $D$ and that the scattered fields solve the Helmholtz equation outside $D$ as well as the radiation condition, we obtain
\be 
\int_{\partial D}\left( u^s_1 \deri{u^i_2}{\nu} - u^i_2 \deri{u^s_1}{\nu} \right) ds = \int_{\partial D}\left( u^s_2 \deri{u^i_1}{\nu} - u^i_1 \deri{u^s_2}{\nu} \right) ds.
\label{echange}
\ee
Now we use the Green's formula on the boundary $\partial D$ for $u^s(\cdot,\hat{x})$: for $z \in \Omega$ and $\hat{x} \in S^{d-1}$,
\[u^s(z,\hat{x}) = \int_{\partial D}\left( u^s(y,\hat{x})\frac{\partial \Phi(y,z)}{\partial\nu(y)} - \frac{\partial u^s(y,\hat{x})}{\partial\nu}\Phi(y,z) \right) ds(y).\]
By applying equation (\ref{echange}) when $u^i_1$ is the plane wave of direction $\hat{x}$ and $u^i_2$ is the point source $\Phi(\cdot,z)$, it follows that
\[u^s(z,\hat{x}) = \int_{\partial D}\left( w^s(y,z)\frac{\partial e^{ik\hat{x}.y}}{\partial\nu(y)} - \frac{\partial w^s(y,z)}{\partial\nu(y)}e^{ik\hat{x}.y}\right) ds(y).\]
Lastly, from the integral representation (\ref{uinf}) and the above equation we obtain
\[\gamma(d)\,u^s(z,\hat{x}) =w^\infty(-\hat{x},z).\] 
\end{proof}
\\
\noindent The second lemma is a density result and does not require the regularity assumption \ref{AsPbinverse} either. Since it is a slightly more general version of Lemma 4 in \cite{bourgeois_haddar}, the proof is omitted.
\begin{lemma}
\label{density}
Let $u(\cdot,\hat{d})$ denote the solution of (\ref{PbInit}) associated to the incident wave $u^i(x)=e^{ik\hat{d}.x}$ and assume that for some $f \in H^{-1}(\partial D)$,
\[<u(\cdot,\hat{d}),f>_{H^1(\partial D),H^{-1}(\partial D)}=0,\quad \forall \hat{d} \in S^{d-1}.\]
Then $f=0$.
\end{lemma} 
\\
\noindent We are now in a position to prove Theorem \ref{1}. 
\begin{proof}[Proof of Theorem \ref{1}] The first step of the proof consists in proving that $D_1=D_2$, following the method of \cite{kirsch_kress,isakov}.
Let us denote $\tilde{\Omega}$ the unbounded connected component of $\mathbb{R}^d \setminus \overline{D_1 \cup D_2}$. From Rellich's lemma and unique continuation,
we obtain that
\be u_1^s(z,\hat{d})=u_2^s(z,\hat{d}),\qquad \forall z \in \tilde{\Omega},\,\forall \hat{d} \in S^{d-1}.\label{rellich}\ee
Using the mixed reciprocity relation of  Lemma \ref{reciprocity}, we get
\[u_1^\infty(-\hat{d},z)=u_2^\infty(-\hat{d},z), \qquad \forall \hat{d} \in S^{d-1},\,\forall z \in \tilde{\Omega},\] 
where $u_1^\infty(\cdot,z)$ and $u_2^\infty(\cdot,z)$ are the far fields associated to the incident field $\Phi(\cdot,z)$ with $z \in \tilde{\Omega}$.
By using again Rellich's lemma and unique continuation, it follows that
\be u_1^s(x,z)=u_2^s(x,z), \qquad \forall (x,z) \in \tilde{\Omega} \times \tilde{\Omega}.\label{12}\ee 
Assume that $D_1 \not\subset D_2$. Since $\mathbb{R}^d \setminus \overline{D_2}$ is connected, there exists some non empty open set $\Gamma_* \subset (\partial D_1 \cap \partial \tilde{\Omega})\setminus \overline{D_2}$. 
We now consider some point $x_* \in \Gamma_*$ and the sequence
\[x_n=x_*+\frac{\nu_1(x_*)}{n}.\]
For sufficiently large $n$, $x_n \in \tilde{\Omega}$.
From (\ref{12}), we hence have by denoting
$P_1 v:=\partial v/\partial \nu + {\rm div}_{\Gamma} (\mu_1 \nabla_{\Gamma} v)+ \lambda_1 v$,
\[P_1 u^s_2(\cdot,x_n)=P_1 u^s_1(\cdot,x_n) \quad {\rm on} \quad \Gamma_*.\]
Using boundary condition on $\partial D_1$ for $u_1=u_1^s+\Phi(\cdot,x_n)$, this implies that
\[P_1 u^s_2(\cdot,x_n)=-P_1 \Phi(\cdot,x_n) \quad {\rm on} \quad \Gamma_*.\]
Using assumption \ref{AsPbinverse} and the fact that $u^s_2$ is smooth in the neighborhood of $\Gamma_*$, we obtain 
\begin{align*}&\lim_{n \rightarrow +\infty} P_1 u^s_2(\cdot,x_n)
=\frac{\partial u^s_2}{\partial \nu}(\cdot,x_*) + \mu_1 \Delta_{\Gamma} u^s_2(\cdot,x_*) + \nabla_{\Gamma} \mu_1 \cdot \nabla_{\Gamma} u^s_2(\cdot,x_*)+ \lambda_1 u^s_2(\cdot,x_*)
\end{align*}
in $L^2(\Gamma_*)$. On the other hand, for $x_1 \in \Gamma_* \setminus\{x_*\}$, we have pointwise convergence
\[
\lim_{n \rightarrow +\infty} P_1 \Phi(x_1,x_n)=\frac{\partial \Phi}{\partial \nu}(x_1,x_*) + \mu_1 \Delta_{\Gamma} \Phi(x_1,x_*) + \nabla_{\Gamma} \mu_1 \cdot \nabla_{\Gamma} \Phi(x_1,x_*)+ \lambda_1 \Phi(x_1,x_*). 
\]
We hence obtain that
\be \frac{\partial \Phi}{\partial \nu}(\cdot,x_*) + {\rm div}_{\Gamma} (\mu_1 \nabla_{\Gamma} \Phi)(\cdot,x_*) + \lambda_1 \Phi(\cdot,x_*) \in  L^2(\Gamma_*).\label{l2}\ee
Now we consider some reals $R_*>r_*>0$ such that $\partial D \cap B(x_*,R_*) \subset \Gamma_*$, a function  $\phi \in C_0^\infty(B(x_*,R_*))$ with $\phi=1$ on $\overline{B(x_*,r_*)}$, and $w^s_*:=\phi\, \Phi(\cdot,x_*)$.
The function $w_*^s$ satisfies 
\be \label{aux} \left\{
\begin{aligned}
& \Delta w_*^s + k^2 w_*^s = f  \quad \text{in} \; \Omega_1\\
& \frac{\partial w_*^s}{\partial \nu} + {\rm div}_{\Gamma} (\mu_1 \nabla_{\Gamma} w_*^s)+ \la_1 w_*^s =g  \quad \text{on} \; \partial D_1\\
& \lim \limits_{R \to \infty} \int_{|x|=R}\left|\deri{w_*^s}{r} - ikw_*^s \right|^2ds =0,
\end{aligned}
\right.
\ee
with $\Omega_1=\mathbb{R}^d \setminus \overline{D_1}$ and
\begin{align*}
f=&(\Delta \phi) \Phi(\cdot,x_*) + 2\nabla \phi \cdot \nabla \Phi(\cdot,x_*),\\
g=&\phi\left(\frac{\partial \Phi}{\partial \nu}(\cdot,x_*) + {\rm div}_{\Gamma} (\mu_1 \nabla_{\Gamma} \Phi)(\cdot,x_*) + \lambda_1 \Phi(\cdot,x_*)\right)\\
&+\Phi(\cdot,x_*)\left(\deri{\phi}{\nu}+\nabla_\Gamma \mu_1\cdot \nabla_\Gamma \phi+\mu_1\Delta_\Gamma \phi\right) +2\mu_1 \nabla_\Gamma \Phi(\cdot,x_*)\cdot \nabla_\Gamma \phi.
\end{align*} 
Since $\phi=1$ in the neighborhood of $x_*$ and by using (\ref{l2}), we have $f \in L^2(\Omega_1)$ and $g \in L^2(\partial D_1)$.
With the help of a variational formulation for the auxiliary problem (\ref{aux}) as in \cite{BoChHa11},
we conclude that $w^s_* \in H^1(B_R \setminus \overline{D_1})$, hence $\Phi (\cdot,x_*) \in H^1(\Omega_1 \cap B(x_*,r_*))$.
Since $\partial D$ is $C^2$, we can find a finite cone $C_*$ of apex $x_*$, angle $\theta_*$, radius $r_*$ and axis directed by $\xi_*=\nu_1(x_*)$, such that $C_* \subset \Omega_1 \cap B(x_*,r_*)$.
Hence $\Phi(\cdot,x_*) \in H^1(C_*)$.\\
In the case $d=3$ (the case $d=2$ is similar), we have
\[\nabla \Phi(\cdot,x_*)=-\frac{e^{ik|x-x_*|}}{4\pi|x-x_*|^2}\left(\frac{1}{|x-x_*|}-ik\right)(x-x_*),\]
and by using spherical coordinates $(r,\theta,\phi)$ centered at $x_*$,
\[\int_{C^*}\frac{dx}{|x-x_*|^4}=\int_0^{r_*}\int_0^{\theta_*} \int_0^{2\pi} \frac{r^2\sin\theta \,drd\theta d\phi}{r^4}=+\infty,\]
which contradicts the fact that $\Phi(\cdot,x_*) \in H^1(C_*)$.
Then $D_1 \subset D_2$. We prove the same way that $D_2 \subset D_1$, and then $D_1=D_2=D$.\\
The second step of the proof consists in proving that $(\la_1,\mu_1)=(\la_2,\mu_2)$.
We set $\la:=\la_1-\la_2$ and $\mu:=\mu_1-\mu_2$. 
From equality (\ref{rellich}), the total fields associated with the plane waves of direction $\hat{d}$ satisfy
\[u(x,\hat{d}):=u_1(x,\hat{d})=u_2(x,\hat{d})\qquad  \forall x \in \mathbb{R}^d \setminus \overline{D},\,\forall \hat{d} \in S^{d-1}.\] 
From the boundary condition on $\partial D$ for $u_1$ and $u_2$ it follows that
\[{\rm div}_{\Gamma} (\mu \nabla_{\Gamma} u(\cdot,\hat{d})) + \lambda u(\cdot,\hat{d})=0  \qquad \rm{on} \; \partial D,\qquad \forall \hat{d} \in S^{d-1}.\]
For some $\phi \in H^1(\partial D)$, by using formula (\ref{bypart}) we obtain
\[<u(\cdot,\hat{d}),{\rm div}_{\Gamma} (\mu \nabla_{\Gamma} \phi) + \lambda \phi>_{H^1(\partial D),H^{-1}(\partial D)}=0,\qquad \forall \hat{d} \in S^{d-1}.\]
With the help of Lemma \ref{density}, we obtain that 
\[{\rm div}_{\Gamma} (\mu \nabla_{\Gamma} \phi) + \lambda \phi=0  \qquad {\rm on} \; \partial D,\quad \forall \phi \in H^1(\partial D).\]
Choosing $\phi=1$ in the above equation leads to $\la=0$.
The above equation also implies that
\[\int_{\partial D}\mu|\nabla_{\Gamma} \phi|^2\,ds=0,\quad \forall \phi \in H^1(\partial D).\]
Assume that $\mu(x_0) \neq 0$ for some $x_0 \in \partial D$, then for example $\Re e(\mu)(x_0) > 0$ without loss of generality. Since $\mu$ is continuous there exists $\eps>0$ such that
$\Re e(\mu)(x) > 0$ for all $x \in \partial D \cap B(x_0,\eps)$. Let us choose $\phi$ as a smooth and compactly supported function in 
$\partial D \cap B(x_0,\eps)$. We obtain that
\[\int_{\partial D \cap B(x_0,\eps)}\Re e(\mu)|\nabla_{\Gamma} \phi|^2\,ds=0,\]
and then $\nabla_{\Gamma}\phi=0$ on $\partial D$, that is $\phi$ is a constant on $\partial D$, which is a contradiction. We hence have $\mu=0$ on $\partial D$, which completes the proof.
\end{proof}
\\
\noindent As illustrated by Theorem \ref{1}, if all plane waves are used as
incident fields, then it is possible to retrieve both the obstacle and the
impedances, with reasonable assumptions on the regularity of the unknowns. In
the sequel, we consider an effective method to retrieve both the obstacle and
the impedances based on a standard steepest descent method. To use such method, one
needs to compute the partial derivative of the far field with respect to the
obstacle shape, the impedances being fixed. This is the aim of next section. 
The adopted approach is the one used in \cite{KrePai} for the Neumann boundary
condition and in \cite{haddar_kress} for the classical impedance boundary
condition with constant $\la$. The computation of the partial derivative with respect to the impedances is already known and given in \cite{BoChHa11}.
\section{Differentiation of far field with respect to the obstacle}
Throughout this section, we assume that the boundary of the obstacle and the impedances are smooth, typically $\partial D$ is $C^4$, $\la \in C^2(\partial D)$ and $\mu \in C^3(\partial D)$, which ensures that the solution to problem (\ref{PbBorne}) belongs to $H^4(\Omega_R)$.
In order to compute the partial derivative of the far field associated to the solution of problem (\ref{PbInit}) with respect to the obstacle,
we consider a perturbed obstacle $D_\eps$ and some impedances
$(\la_\eps,\mu_\eps)$ that correspond with the impedances $(\la,\mu)$ 
composed with the mapping  $\partial D_\eps \mapsto \partial D$.\\
More precisely, we consider some mapping $\eps \in C^{1,\infty}(\mathbb{R}^d,\mathbb{R}^d)$ with $C^{1,\infty}:=C^1 \cap W^{1,\infty}$ equipped with the norm $||\eps||:=||\eps||_{W^{1,\infty}(\mathbb{R}^d,\mathbb{R}^d)}$.
From \cite[section 5.2.2]{HenPie}, if we assume that $||\eps||<1$, the mapping $f_\eps:=Id + \eps$ is a $C^1$--diffeomorphism of $\mathbb{R}^d$. 
The perturbed obstacle $D_\eps$ is defined by 
\[\partial D_\eps=\{x + \eps(x),\,x \in \partial D\},\]
while the  impedances $(\la_\eps,\mu_\eps)$ are defined on $\partial D_\eps$ by
\begin{equation}\la_\eps =\la \circ f_\eps^{-1},\quad \mu_\eps =\mu \circ f_\eps^{-1}.\label{transport}\end{equation}
We now define the partial derivative of the far field with respect to the
obstacle shape.
\begin{definition}
\label{frechet}
We say that the far field operator
$T : (\lambda,\mu,\partial{D}) \rightarrow u^{\infty}$
is differentiable with respect to $\partial D$ if there exists a continuous linear operator $T'_{\lambda,\mu}(\partial D) : C^{1,\infty}(\mathbb{R}^d,\mathbb{R}^d) \rightarrow L^2 (S^{d-1})$
and a function $o(||\eps||): C^{1,\infty}(\mathbb{R}^d,\mathbb{R}^d) \rightarrow L^2 (S^{d-1})$ such that 
\[T(\la_\eps,\mu_\eps,\partial D_\eps)-T(\la,\mu,\partial D)=T_{\lambda,\mu}'(\partial D)\cdot\eps + o(||\eps||),\]
where $\la_\eps$ and $\mu_\eps$ are defined by (\ref{transport}) and
$\lim_{||\eps|| \rightarrow 0} o(||\eps||)/||\eps||=0$
in $L^2 (S^{d-1})$.
\end{definition}
\begin{remark}
Note that if $\la$ and $\mu$ are constants, the above definition coincides with the classical notion of Fr\'echet differentiability with respect to an obstacle. 
\end{remark}
\\
\noindent
Now we denote by $u_\eps$ the solution of problem 
(\ref{PbInit}) with obstacle $D_\eps$ instead of obstacle $D$ and impedances $(\la_\eps,\mu_\eps)$ instead of impedances $(\la,\mu)$.
We assume in addition that $\overline{D} \subset D_\eps$.
We then have the following integral representation for $u_\eps^s-u^s$:
\begin{lemma}
\label{int1}
For $x \in \mathbb{R}^d \setminus \overline{D}_\eps$,
\[u_\eps^s(x)- u^s(x)=\int_{\partial D_\eps} u_\eps \left\{\frac{\partial w}{\partial \nu_\eps}(\cdot,x) + {\rm div}_{\Gamma_\eps} (\mu_\eps \nabla_{\Gamma_\eps} w)(\cdot,x) + \lambda_\eps w(\cdot,x)\right\}ds_\eps,\]
where $w(\cdot,x)$ is the solution of problem (\ref{PbInit}) with incident wave $\Phi(\cdot,x)$.
\end{lemma}
\begin{proof}
Let $x \in \mathbb{R}^d \setminus \overline{D_\eps}$.
% By using the Green's integral theorem inside $D$ for plane wave $u^i$ and point source $\Phi(\cdot,x)$, we have
% \[\int_{\partial D}\left(u^i\deri{\Phi}{\nu}(\cdot,x)-\deri{u^i}{\nu}\Phi(\cdot,x)\right)ds=0,\]
% then obtain the
We recall the representation formula for the scattered field
\be u^s(x)=\int_{\partial D}\left(u\deri{\Phi}{\nu}(\cdot,x)-\deri{u}{\nu}\Phi(\cdot,x)\right)ds.\label{repres}\ee
By using the boundary condition on $\partial D$ in problem (\ref{PbInit}) for functions $u$ and $w(\cdot,x)=w^s(\cdot,x) + \Phi(\cdot,x)$ and by using formula (\ref{bypart}), we obtain 
\[\int_{\partial D}\left(u\deri{\Phi}{\nu}(\cdot,x)-\deri{u}{\nu}\Phi(\cdot,x)\right)ds= -\int_{\partial D}\left(u\deri{w^s}{\nu}(\cdot,x)-\deri{u}{\nu}w^s(\cdot,x)\right)ds.\]
From the two above equations, using  the Green's integral theorem outside $D$ and the radiation condition for $u^s$ and $w^s(\cdot,x)$,
\[u^s(x)=\int_{\partial D}\left(\deri{u^i}{\nu}w^s(\cdot,x)- u^i\deri{w^s}{\nu}(\cdot,x)\right)ds.\]
We now use the Green's integral theorem in $D_\eps \setminus \overline{D}$ for functions $u^i$ and $w^s(\cdot,x)$ and find
\[u^s(x)=\int_{\partial D_\eps}\left(\deri{u^i}{\nu_\eps}w^s(\cdot,x)- u^i\deri{w^s}{\nu_\eps}(\cdot,x)\right)ds_\eps.\]
Using again the Green's integral theorem outside $D_\eps$ and the radiation condition for $u^s_\eps$ and $w^s(\cdot,x)$, 
\[u^s(x)=\int_{\partial D_\eps}\left(\deri{u_\eps}{\nu_\eps}w^s(\cdot,x)- u_\eps \deri{w^s}{\nu_\eps}(\cdot,x)\right)ds_\eps.\]
The boundary condition satisfied by $u_\eps$ on $\partial D_\eps$ implies
\[-u^s(x)=\int_{\partial D_\eps}u_\eps\left(\frac{\partial w^s}{\partial \nu_\eps}(\cdot,x) + {\rm div}_{\Gamma_\eps} (\mu_\eps \nabla_{\Gamma_\eps} w^s)(\cdot,x) + \lambda_\eps w^s(\cdot,x)\right)ds_\eps.\]
Lastly, we use formula (\ref{repres}) for $u_\eps$ and $D_\eps$, as well as the boundary condition of $u_\eps$ on $\partial D_\eps$, and obtain
that for $x \in \mathbb{R}^d \setminus \overline{D_\eps}$,
\[u_\eps^s(x)=\int_{\partial D_\eps}u_\eps\left(\frac{\partial \Phi}{\partial \nu_\eps}(\cdot,x) + {\rm div}_{\Gamma_\eps} (\mu_\eps \nabla_{\Gamma_\eps} \Phi)(\cdot,x) + \lambda_\eps \Phi(\cdot,x)\right)ds.\]
We complete the proof by adding the two last equations, given $w(\cdot,x)=w^s(\cdot,x)+\Phi(\cdot,x)$.
\end{proof}

\noindent Using the continuity of the solution with respect to the shape one
can replace  $u_\eps$ by $u$ in the integral representation of Lemma \ref{int1}
up to  $\mathcal{O}(||\eps||^2)$ terms. In the following, we make use of the
notation
\[J_\eps:=|{\rm det}(\nabla f_\eps)|,\quad J_\eps^\nu:=J_\eps|(\nabla f_\eps)^{-T}\nu|,\quad P_\eps:=(\nabla f_\eps)^{-1}(\nabla f_\eps)^{-T},
\]
where ${\rm det}(B)$ stands for the determinant of matrix $B$, while $B^{-T}$ stands for the transposition of the inverse of invertible matrix $B$.
\begin{lemma}
\label{int2} Let $w$ be as in Lemma \ref{int1}. Then
\[u_\eps^s(x)- u^s(x)=\int_{\partial D_\eps} u \left\{\frac{\partial w}{\partial \nu_\eps}(\cdot,x) + {\rm div}_{\Gamma_\eps} (\mu_\eps \nabla_{\Gamma_\eps} w)(\cdot,x) + \lambda_\eps w(\cdot,x)\right\}ds_\eps + \mathcal{O}(||\eps||^2),\]
uniformly for $x$ in each compact subset $K \subset \mathbb{R}^d \setminus \overline{D}$.
\end{lemma}
\begin{proof}
We first observe that
\begin{align}\int_{\partial D_\eps}(&u_\eps-u)\left\{\frac{\partial w}{\partial \nu_\eps}(\cdot,x) + {\rm div}_{\Gamma_\eps} (\mu_\eps \nabla_{\Gamma_\eps} w)(\cdot,x) + \lambda_\eps w(\cdot,x)\right\}ds_\eps \nonumber \\
=\int_{\partial D_\eps}(u_\eps-u)&\deri{w}{\nu_\eps}\,ds_\eps-\int_{\partial D_\eps}\mu_\eps \nabla_{\Gamma_\eps}(u_\eps-u)\cdot\nabla_{\Gamma_\eps} w\,ds_\eps+ \int_{\partial D_\eps}\la_\eps (u_\eps-u)w\,ds_\eps. \label{integral}
\end{align}
We shall only consider the second integral on the right hand side of the last
equality. The other two integrals are simpler terms and can be treated in
a similar way (see for instance \cite{BoChHa11report} for a detailed proof).
By denoting $\tilde{u}_\eps=u_\eps \circ f_\eps$, $\hat{u}_\eps=u \circ f_\eps$ and $\hat{w}_\eps=w \circ f_\eps$,
we have
\[\int_{\partial D_\eps}\mu_\eps \nabla_{\Gamma_\eps}(u_\eps-u)\cdot\nabla_{\Gamma_\eps} w\,ds_\eps
=\int_{\partial D_\eps}(\mu \circ f_\eps^{-1}) \nabla_{\Gamma_\eps}((\tilde{u}_\eps-\hat{u}_\eps)\circ f_\eps^{-1})\cdot\nabla_{\Gamma_\eps} (\hat{w}_\eps\, \circ f_\eps^{-1})\,ds_\eps.\]
 For $z \in H^1(\partial D)$, $x\in \partial D$ and $x_\eps=f_\eps(x)$,
\[\nabla_{\Gamma_\eps} (z \circ f_\eps^{-1})(x_\eps)=(\nabla f_\eps(x))^{-T}\nabla_{\Gamma}\,z(x).\]
(See e.g. \cite[proof of Lemma 3.4]{BoChHa11}.) Consequently, the change of variable $x_\eps=f_\eps(x)$ in the integral (see \cite[Proposition 5.4.3]{HenPie}) implies
\[\int_{\partial D_\eps}\mu_\eps \nabla_{\Gamma_\eps}(u_\eps-u)\cdot\nabla_{\Gamma_\eps} w\,ds_\eps
=\int_{\partial D} \mu \nabla_{\Gamma}(\tilde{u}_\eps-\hat{u}_\eps)\cdot P_\eps \cdot \nabla_{\Gamma}\hat{w}_\eps J_\eps^\nu\,ds,\]
and
\begin{align*}\int_{\partial D_\eps}\mu_\eps &\nabla_{\Gamma_\eps}(u_\eps-u)\cdot\nabla_{\Gamma_\eps} w\,ds_\eps-
\int_{\partial D}\mu \nabla_{\Gamma}(\tilde{u}_\eps-\hat{u}_\eps)\cdot\nabla_{\Gamma} w\,ds\\
&=\int_{\partial D}\mu\nabla_{\Gamma}(\tilde{u}_\eps-\hat{u}_\eps)\cdot(J^\nu_\eps P_\eps\cdot\nabla_{\Gamma} \hat{w}_\eps-\nabla_{\Gamma}w)\,ds.
\end{align*}
Then
\begin{align*}\left|\int_{\partial D_\eps}\mu_\eps \nabla_{\Gamma_\eps}\right.&\left.(u_\eps-u)\cdot\nabla_{\Gamma_\eps} w\,ds_\eps-
\int_{\partial D}\mu \nabla_{\Gamma}(\tilde{u}_\eps-\hat{u}_\eps)\cdot\nabla_{\Gamma} w\,ds\right|\\
&\leq ||\mu||_{L^\infty(\partial D)}||\tilde{u}_\eps-\hat{u}_\eps||_{H^1(\partial D)}V_\eps(x)
\end{align*}
with
\[V_\eps(x)=||(J^\nu_\eps P_\eps\cdot\nabla_{\Gamma} (w \circ f_\eps)-\nabla_{\Gamma} w)(\cdot,x)||_{L^2(\partial D)}.\]
By using the fact that 
\[J_\eps^\nu=1+\mathcal{O}(||\eps||),\quad P_\eps=Id(1+\mathcal{O}(||\eps||)),\]
we conclude that
$V_\eps(x)=\mathcal{O}(||\eps||)$,
uniformly for $x$ in some compact subset $K \subset \mathbb{R}^d \setminus \overline{D}$.\\
On the other hand,
\[||\tilde{u}_\eps-\hat{u}_\eps||_{H^1(\partial D)} \leq ||\tilde{u}_\eps-u||_{H^1(\partial D)} + ||\hat{u}_\eps-u||_{H^1(\partial D)}.\]
We have 
\[||\tilde{u}_\eps-u||_{H^1(\partial D)}=\mathcal{O}(||\eps||),\quad ||\hat{u}_\eps-u||_{H^1(\partial D)}= \mathcal{O}(||\eps||).\]
The first estimate is a consequence of Theorem 3.1 in \cite{BoChHa11}, that is continuity of the solution of problem (\ref{PbBorne}) with respect to the scatterer $D$. The second one comes from the fact that $\hat{u}_\eps=u \circ f_\eps$.\\
We hence conclude that 
\[\int_{\partial D_\eps}\mu_\eps \nabla_{\Gamma_\eps}(u_\eps-u)\cdot\nabla_{\Gamma_\eps} w(\cdot,x)\,ds_\eps-
\int_{\partial D}\mu \nabla_{\Gamma}(\tilde{u}_\eps-\hat{u}_\eps)\cdot\nabla_{\Gamma} w(\cdot,x)\,ds
=\mathcal{O}(||\eps||^2),
\]
uniformly for $x$ in some compact subset $K \subset \mathbb{R}^d \setminus \overline{D}$, and we treat the other two integrals of (\ref{integral}) similarly.
We remark that due to boundary condition satisfied by $w(\cdot,x)$ on $\partial D$, we have
\begin{align*}&0=\int_{\partial D}(\tilde{u}_\eps-\hat{u}_\eps)\deri{w}{\nu}\,ds-\int_{\partial D}\mu \nabla_{\Gamma}(\tilde{u}_\eps-\hat{u}_\eps)\cdot\nabla_{\Gamma} w\,ds+ \int_{\partial D}\la(\tilde{u}_\eps-\hat{u}_\eps)w\,ds\\
=\int_{\partial D_\eps}&(u_\eps-u)\frac{\partial w}{\partial \nu_\eps}\,ds_\eps- \int_{\partial D_\eps}\mu_\eps \nabla_{\Gamma_\eps}(u_\eps-u)\cdot \nabla_{\Gamma_\eps} w\,ds_\eps + \int_{\partial D_\eps}\lambda_\eps(u_\eps-u)w\,ds_\eps+
\mathcal{O}(||\eps||^2),\end{align*}
and conclude that
\[\int_{\partial D_\eps}(u_\eps-u)\left\{\frac{\partial w}{\partial \nu_\eps}(\cdot,x) + {\rm div}_{\Gamma_\eps} (\mu_\eps \nabla_{\Gamma_\eps} w)(\cdot,x) + \lambda_\eps w(\cdot,x)\right\}ds_\eps =
\mathcal{O}(||\eps||^2),\]
which, combined with Lemma \ref{int1}, gives the desired result.
\end{proof}

\noindent The next step is to transform the surface representation of Lemma
\ref{int2} into a representation that uses $\partial D$ instead of $\partial
D_\eps$ by using the divergence theorem. We therefore need to extend the definitions of some surface quantifies, essentially the outward normal $\nu$ on $\partial D$ and the surface gradient $\nabla_{\Gamma}$ inside the volumetric domain $D_\eps \setminus \overline{D}$. 
In this view, for $x_0 \in \partial D$, by definition of a domain of class $C^1$ there exist a function $\phi$ of class $C^1$ and two open sets $U \subset \mathbb{R}^{d-1}$ and $V \subset \mathbb{R}^d$
which are neighborhood of $0$ and $x_0$ respectively, such that $\phi(0)=x_0$ and 
\[\partial D \cap V=\{\phi(\xi)\,;\xi \in U\}.\]
Defining now for $t \in [0,1]$,
\[f_t:= Id + t\eps,\quad \phi_t:=f_t \circ \phi,\]
$\phi_t$ is a parametrization of $\partial D_t=(Id + t\eps)(\partial D)$,
and hence the tangential vectors of $\partial D_t$ at $x^t_0=f_t(x_0)$ are
\be e_j^t=\deri{\phi_t}{\xi_j}=(Id + t \nabla \eps)\deri{\phi}{\xi_j}=(Id + t \nabla \eps)e_j,\quad {\rm for}\quad j=1,d-1.\label{tangent}\ee
We define the covariant basis $(e^{i}_t)$ of $\partial D_t$ at point $x^t_0$ (see for example \cite[section 2.5]{book-nedelec}) by
\be e^{i}_t\cdot e_j^t=\delta_j^i,\quad {\rm for}\quad i,j=1,d-1. \label{covariant}\ee
With these definitions, the outward normal of $\partial D_t$ at point $x^t_0$ is given by
\[\nu_t=\frac{e_1^t \times e_2^t}{|e_1^t \times e_2^t|},\]
while the tangential gradient of function $w \in H^1(\partial D_t)$ is given, denoting $\tilde{w}_t=w\,\circ\, \phi_t$, by
\be \nabla_{\Gamma_t}w(x^t_0)=\sum_{i=1}^{d-1}\deri{\tilde{w}_t}{\xi_i}(0)e^{i}_t.\label{surface}\ee
It is hence possible to consider in domain $D_\eps \setminus \overline{D}$ an extended outward normal $\nu_t$ 
and an extended tangential gradient $\nabla_{\Gamma_t}w$ by using parametrization $(\xi_i,t)$ for $i=1,d-1$.
In the same spirit, the impedances $(\la,\mu)$ are extended to $(\la_t,\mu_t)$, by
\[\la_t:=\la \circ f_t^{-1},\quad \mu_t:=\mu \circ f_t^{-1}.\]
We are now in a position to transform the integral representation of $u_\eps^s-
u^s$  in Lemma \ref{int2} into an integral representation that uses $\partial D$.
We have the following proposition.
\begin{lemma}
\label{int3}
Let $w$ be as in Lemma \ref{int1}. Then, \begin{align*}&u_\eps^s(x)- u^s(x)\\
=\int_{\partial D} (\eps\cdot\nu)\,{\rm div}\left(\right.&\left.-\mu_t \nabla_{\Gamma_t}u\cdot\nabla_{\Gamma_t} w(\cdot,x) \nu_t  + u \nabla w(\cdot,x) + \la_t uw(\cdot,x) \nu_t\right)|_{t=0}ds + \mathcal{O}(||\eps||^2),\end{align*}
uniformly for $x$ in some compact subset $K \subset \mathbb{R}^d \setminus \overline{D}$.
\end{lemma}
\begin{proof}
The proof relies on the Green's integral theorem and on a change of variable.
We have by using the extension of fields as described above and noticing that $\nu^0=\nu$ and $\nu^1=\nu_\eps$,
\begin{align*}
\int_{\partial D_\eps} &\left\{u\frac{\partial w}{\partial \nu_\eps} - \mu_\eps \nabla_{\Gamma_\eps} u \cdot\nabla_{\Gamma_\eps} w + \lambda_\eps uw\right\}ds_\eps 
-\int_{\partial D} \left\{u\frac{\partial w}{\partial \nu} - \mu \nabla_{\Gamma} u\cdot\nabla_{\Gamma} w + \lambda uw\right\}ds\\
&=\int_{D_\eps \setminus \overline{D}} {\rm div}\left\{u \nabla w - \mu_t (\nabla_{\Gamma_t} u\cdot\nabla_{\Gamma_t} w)\, \nu_t+ \lambda_t uw \,\nu_t\right\}dx\\
&=\int_{\partial D}\int_0^1 {\rm div}\left\{u \nabla w - \mu_t (\nabla_{\Gamma_t} u\cdot\nabla_{\Gamma_t} w)\, \nu_t+ \lambda_t uw \,\nu_t\right\}(\eps\cdot\nu)\,dtds + \mathcal{O}(||\eps||^2).\end{align*}
Here we have used the change of variable $(x_{\partial D},t) \rightarrow
x_{\partial D}+t\eps(x_{\partial D})$ for $x_{\partial D} \in \partial D$ and
$t \in [0,1]$ and
the determinant of the associated Jacobian matrix is $(\eps \cdot \nu) + t\, \mathcal{O}(||\eps||^2)$.
Lastly, by using a first order approximation of the integrand as in \cite{KrePai}, we obtain 
\begin{align*}
\int_{\partial D_\eps} &\left\{u\frac{\partial w}{\partial \nu_\eps} - \mu_\eps \nabla_{\Gamma_\eps} u\cdot\nabla_{\Gamma_\eps} w + \lambda_\eps uw\right\}ds_\eps 
-\int_{\partial D} \left\{u\frac{\partial w}{\partial \nu} - \mu \nabla_{\Gamma} u\cdot\nabla_{\Gamma} w + \lambda uw\right\}ds\\
&=\int_{\partial D}(\eps\cdot\nu){\rm div}\left\{u \nabla w - \mu_t (\nabla_{\Gamma_t} u\cdot\nabla_{\Gamma_t}w)\, \nu_t+ \lambda_t uw \,\nu_t\right\}|_{t=0}ds +\mathcal{O}(||\eps||^2).\end{align*}
The proof is completed by combining the boundary condition satisfied by $u$ on $\partial D$, formula (\ref{bypart}) and the result of Lemma \ref{int2}.
\end{proof}

\noindent The remainder of the section consists in expressing the trace of the divergence term.
In order to do that we need the following technical lemma, the proof of which is postponed in an appendix.
\begin{lemma}
\label{lem}
Let $\lambda$ be a $C^1(\partial D)$ function and define $\lambda_t := \lambda \circ
f_t^{-1}$ and let $u$ and $w$ be in $H^4(B_R\setminus \overline{D})$. Then the following
identities hold on $\partial D$,
\begin{align*}(\eps\cdot\nu)(\nabla
  \la_t\cdot\nu_t)|_{t=0}=-(\nabla_{\Gamma}\la\cdot\eps),
\end{align*}
\begin{align*}
({\rm div} \nu_t)|_{t=0}={\rm div}_{\Gamma} \nu,\end{align*}
\begin{multline*}(\eps\cdot\nu)\nabla(\nabla_{\Gamma_t}u\cdot\nabla_{\Gamma_t}w)\cdot\nu_t|_{t=0}=-\eps_{\Gamma}\cdot\nabla_{\Gamma}(\nabla_{\Gamma}u\cdot\nabla_{\Gamma}w)
+ \nabla_{\Gamma}(\eps_{\Gamma}\cdot \nabla_{\Gamma}u+(\nabla u\cdot\nu)(\eps\cdot\nu))\cdot\nabla_{\Gamma}w \\+ \nabla_{\Gamma}u\cdot\nabla_{\Gamma}(\nabla_{\Gamma}w\cdot\eps_{\Gamma}+(\nabla w\cdot\nu)(\eps\cdot\nu))
-\nabla_{\Gamma}u\cdot(\nabla \eps+ (\nabla \eps)^T)\cdot\nabla_{\Gamma}w,
\end{multline*}
where  we have set $\eps_{\Gamma}:=\eps-(\eps\cdot\nu)\nu$.
\end{lemma}

\noindent In order to simplify the presentation we split the computation of the divergence term in Lemma \ref{int3} into two terms that we treat separately.
\begin{proposition} Let $\mu$ be a $C^1(\partial D)$ function and define $\mu_t := \mu \circ
f_t^{-1}$ and let $u$ and $w$ be in $H^4(B_R\setminus \overline{D})$. Then the following
identity holds on $\partial D$,
\label{div2}
\begin{align*}
%\begin{array}{lcl}
(\eps\cdot\nu)\,{\rm div}\left(\mu_t\nabla_{\Gamma_t}u\cdot\nabla_{\Gamma_t}w\nu_t\right)|_{t=0}=& \displaystyle
-(\nabla_{\Gamma}\mu\cdot\eps) (\nabla_{\Gamma}u\cdot\nabla_{\Gamma}w)+\mu (\eps\cdot\nu)(\nabla_{\Gamma}u\cdot\nabla_{\Gamma}w)({\rm div}_{\Gamma}\nu)\\
& \displaystyle+ \mu(\eps\cdot\nu) \nabla_{\Gamma}(\nabla u\cdot\nu)\cdot\nabla_{\Gamma}w+\mu(\eps\cdot\nu)\nabla_{\Gamma}u\cdot\nabla_{\Gamma}(\nabla w\cdot\nu)\\
& \displaystyle+ \mu (\nabla u\cdot\nu) \nabla_{\Gamma}(\eps\cdot\nu)\cdot\nabla_{\Gamma}w + \mu (\nabla w\cdot\nu)\nabla_{\Gamma}(\eps\cdot\nu)\cdot\nabla_{\Gamma}u\\
 & \displaystyle-2\mu(\eps\cdot\nu)(\nabla_{\Gamma}u\cdot\nabla_{\Gamma}\nu\cdot\nabla_{\Gamma}w).
%\end{array}
\end{align*}
\end{proposition}
\begin{proof}
Using the chain rule,
\begin{multline*}{\rm div}\left(\mu_t\nabla_{\Gamma_t}u\cdot\nabla_{\Gamma_t}w\nu_t\right)|_{t=0}\\
= (\nabla \mu_t\cdot\nu_t)|_{t=0}(\nabla_{\Gamma}u\cdot\nabla_{\Gamma}w)+\mu (\nabla_{\Gamma}u\cdot\nabla_{\Gamma}w)({\rm div \nu_t})|_{t=0}+\mu\nabla (\nabla_{\Gamma_t}u\cdot\nabla_{\Gamma_t}w)\cdot\nu_t|_{t=0}.\end{multline*}
By using Lemma \ref{lem},
we obtain that 
\begin{multline}(\eps\cdot\nu){\rm div}\left(\mu_t\nabla_{\Gamma_t}u\cdot\nabla_{\Gamma_t}w\nu_t\right)|_{t=0}=-(\nabla_{\Gamma}\mu\cdot\eps)(\nabla_{\Gamma}u\cdot\nabla_{\Gamma}w)+\mu (\eps\cdot\nu)(\nabla_{\Gamma}u\cdot\nabla_{\Gamma}w)({\rm div_{\Gamma}\nu})\\
-\mu\,
\eps_{\Gamma}\cdot\nabla_{\Gamma}(\nabla_{\Gamma}u\cdot\nabla_{\Gamma}w)-\mu\nabla_{\Gamma}u\cdot(\nabla \eps+ (\nabla
 \eps)^T)\cdot\nabla_{\Gamma}w \\
 +
\mu\nabla_{\Gamma}(\eps_{\Gamma}\cdot\nabla_{\Gamma}u +(\nabla u\cdot\nu)(\eps\cdot\nu))\cdot\nabla_{\Gamma}w +
 \mu\nabla_{\Gamma}u\cdot\nabla_{\Gamma}(\nabla_{\Gamma}w\cdot\eps_{\Gamma}+(\nabla
 w\cdot\nu)(\eps\cdot\nu)).  \label{toot1}\end{multline}
For a surface vector $a_\Gamma$, we denote by $\nabla_\Gamma a_\Gamma$ the $(d-1) \times (d-1)$ tensor defined by
\[ \nabla_\Gamma a_\Gamma \cdot e_j=\deri{a_\Gamma}{\xi_j},\quad j=1,..,d-1.\]
Using the algebraic identity $\nabla_\Gamma(a_\Gamma\cdot b_\Gamma)=(\nabla_\Gamma a_\Gamma)^T\cdot b_\Gamma+a_\Gamma \cdot\nabla_\Gamma b_\Gamma$, the third line of (\ref{toot1}) can be expressed as
\begin{multline} \mu\nabla_{\Gamma}(\eps_{\Gamma}\cdot\nabla_{\Gamma}u+ (\nabla u\cdot\nu)(\eps\cdot\nu))\cdot\nabla_{\Gamma}w + \mu\nabla_{\Gamma}u\cdot\nabla_{\Gamma}(\nabla_{\Gamma}w\cdot\eps_{\Gamma}+ (\nabla w\cdot\nu)(\eps\cdot\nu))\\
=\mu\,
\eps_{\Gamma}\cdot\nabla_{\Gamma}(\nabla_{\Gamma}u)\cdot\nabla_{\Gamma}w+\mu
\nabla_{\Gamma}u\cdot\nabla_{\Gamma}\eps_{\Gamma}\cdot\nabla_{\Gamma}w +\mu (\eps\cdot\nu)\nabla_{\Gamma}(\nabla u\cdot\nu)\cdot\nabla_{\Gamma}w \\
+ \mu(\nabla u\cdot\nu)\nabla_{\Gamma}(\eps\cdot\nu)\cdot\nabla_{\Gamma}w
+ \mu \nabla_{\Gamma}u\cdot\nabla_{\Gamma}(\nabla_{\Gamma}w)\cdot\eps_{\Gamma}+\mu \nabla_{\Gamma}u\cdot(\nabla_{\Gamma}\eps_{\Gamma})^{T}\cdot\nabla_{\Gamma}w\\
- \mu (\eps\cdot\nu) \nabla_{\Gamma}u\cdot\nabla_{\Gamma}(\nabla w\cdot\nu)+\mu (\nabla w\cdot\nu)\nabla_{\Gamma}u\cdot\nabla_{\Gamma}(\eps\cdot\nu). \label{toot2}\end{multline}
Plugging (\ref{toot2}) into (\ref{toot1}) and using the identity
\[\eps_{\Gamma}\cdot\nabla_{\Gamma}(\nabla_{\Gamma}u\cdot\nabla_{\Gamma}w)=\eps_{\Gamma}\cdot\nabla_{\Gamma}(\nabla_{\Gamma}u)\cdot\nabla_{\Gamma}w+\nabla_{\Gamma}u\cdot\nabla_{\Gamma}(\nabla_{\Gamma}w)\cdot\eps_{\Gamma},\]
one ends up with
\begin{align*}(\eps\cdot\nu){\rm div}&\left(\mu_t\nabla_{\Gamma_t}u\cdot\nabla_{\Gamma_t}w\nu_t\right)|_{t=0}=-(\nabla_{\Gamma}\mu\cdot\eps)(\nabla_{\Gamma}u\cdot\nabla_{\Gamma}w)+\mu (\eps\cdot\nu)(\nabla_{\Gamma}u\cdot\nabla_{\Gamma}w)({\rm div_{\Gamma}\nu})\\
&+ \mu (\eps\cdot\nu)\nabla_{\Gamma}(\nabla u\cdot\nu)\cdot\nabla_{\Gamma}w + \mu (\nabla u\cdot\nu)\nabla_{\Gamma}(\eps\cdot\nu)\cdot\nabla_{\Gamma}w\\
&+ \mu (\eps\cdot\nu) \nabla_{\Gamma}u\cdot\nabla_{\Gamma}(\nabla w\cdot\nu)+\mu (\nabla w\cdot\nu)\nabla_{\Gamma}u\cdot\nabla_{\Gamma}(\eps\cdot\nu)\\
&-\mu \nabla_{\Gamma}u\cdot(\nabla_{\Gamma} \eps+ (\nabla_{\Gamma} \eps)^T)\cdot\nabla_{\Gamma}w  +\mu \nabla_{\Gamma}u\cdot(\nabla_{\Gamma}\eps_{\Gamma}+(\nabla_{\Gamma}\eps_{\Gamma})^{T})\cdot\nabla_{\Gamma}w.\end{align*}
Now we need to evaluate $\nabla_{\Gamma}\eps-\nabla_{\Gamma}\eps_{\Gamma}$. Since $\eps-\eps_{\Gamma}=(\eps\cdot\nu)\nu$, we have
\[\nabla_{\Gamma}\eps-\nabla_{\Gamma}\eps_{\Gamma}=\nabla_{\Gamma}((\eps\cdot\nu)\nu)=(\eps\cdot\nu)\nabla_{\Gamma}\nu +\nu \otimes \nabla_{\Gamma}(\eps\cdot\nu),\]
where for a surface field $a$, we denote by $\nu \otimes \nabla_\Gamma a$ the $d \times d$ tensor $M$ defined by
\[M\cdot e_j = \deri{a}{\xi_j} \nu\quad j=1,..,d-1,\quad M\cdot \nu=0.\] 
This implies in particular
\[\nabla_{\Gamma}u\cdot\nabla_{\Gamma}\eps_{\Gamma}\cdot\nabla_{\Gamma}w-
\nabla_{\Gamma}u\cdot\nabla_{\Gamma}\eps\cdot\nabla_{\Gamma}w=-(\eps\cdot\nu)\nabla_{\Gamma}u\cdot\nabla_{\Gamma}\nu\cdot\nabla_{\Gamma}w.\]
Since the tensor $\nabla_{\Gamma}\nu$ is symmetric (see for example \cite[Theorem 2.5.18]{book-nedelec}), we also obtain
\[\nabla_{\Gamma}u\cdot(\nabla_{\Gamma}\eps_{\Gamma})^T\cdot\nabla_{\Gamma}w-
\nabla_{\Gamma}u\cdot(\nabla_{\Gamma}\eps)^T\cdot\nabla_{\Gamma}w=-(\eps\cdot\nu)\nabla_{\Gamma}u\cdot\nabla_{\Gamma}\nu\cdot\nabla_{\Gamma}w.\]
We finally arrive at 
\begin{align*}&(\eps\cdot\nu){\rm div}\left(\mu_t\nabla_{\Gamma_t}u\cdot\nabla_{\Gamma_t}w\nu_t\right)|_{t=0}\\
=-(\nabla_{\Gamma}\mu&\cdot\eps)(\nabla_{\Gamma}u\cdot\nabla_{\Gamma}w)+\mu (\eps\cdot\nu)(\nabla_{\Gamma}u\cdot\nabla_{\Gamma}w)({\rm div_{\Gamma}\nu})\\
+\mu (\eps\cdot\nu)&\nabla_{\Gamma}(\nabla u\cdot\nu)\cdot\nabla_{\Gamma}w + \mu (\nabla u\cdot\nu)\nabla_{\Gamma}(\eps\cdot\nu)\cdot\nabla_{\Gamma}w\\
+ \mu (\eps\cdot\nu) \nabla_{\Gamma}u\cdot\nabla_{\Gamma}(\nabla w&\cdot\nu)+\mu (\nabla w\cdot\nu)\nabla_{\Gamma}u\cdot\nabla_{\Gamma}(\eps\cdot\nu)
-2\mu(\eps\cdot\nu)\nabla_{\Gamma}u\cdot\nabla_{\Gamma}\nu\cdot\nabla_{\Gamma}w,\end{align*}
which completes the proof.
\end{proof}
\begin{proposition}
\label{div1}
Let $u$ and $w(\cdot, x)$ be as in Lemma \ref{int3}. Then the following
identity holds on $\partial D$,
\begin{multline*}(\eps\cdot\nu)\,{\rm div}\left(u \nabla w(\cdot,x) + \la_t
    uw(\cdot,x) \nu_t\right)|_{t=0}=-(\nabla_{\Gamma}\la\cdot\eps) uw(\cdot,x) \\
+ (\eps\cdot\nu)\left(\nabla_{\Gamma}u\cdot\nabla_{\Gamma}w(\cdot,x)+M_\mu
  u\,M_\mu w(\cdot,x)  -(k^2+\la^2-\la ({\rm
    div}_{\Gamma}\nu)uw(\cdot,x))\right),
\end{multline*}
where we have used the short notation $M_\mu :={\rm div}_{\Gamma}(\mu \nabla_{\Gamma}\cdot )$.
\end{proposition}
\begin{proof}
We first observe that
\[{\rm div}(u \nabla w+\la_t uw \nu_t)=\nabla u\cdot\nabla w + u \Delta w + uw(\nabla \la_t\cdot\nu_t)+\la_t w\nabla u\cdot\nu_t + \la_t u\nabla w\cdot\nu_t+ \la_t uw({\rm div}\nu_t).\]
By using the equation $\Delta w+k^2w=0$ outside $D$ and the decomposition of gradient into its normal and tangential parts, we obtain
\begin{align*}{\rm div}(u \nabla w+\la_t uw \nu_t)|_{t=0}=&\nabla_{\Gamma} u\cdot\nabla_{\Gamma} w + (\nabla u\cdot\nu)(\nabla w\cdot\nu) - (k^2-\la ({\rm div} \nu_t)|_{t=0})uw \\
&+ uw (\nabla \la_t\cdot\nu_t)|_{t=0}+\la (\nabla u\cdot\nu)w + \la u(\nabla w\cdot\nu).\end{align*}
We can now replace $\nabla u\cdot\nu$ and $\nabla w\cdot\nu$ by $(-M_\mu u-\la u)$ and $(-M_\mu w-\la w)$ respectively, which leads to
\begin{align*}{\rm div}(u \nabla w+\la_t uw \nu_t)|_{t=0}= &\nabla_{\Gamma} u\cdot\nabla_{\Gamma} w + M_\mu u \,M_\mu w-(k^2+\la^2-\la ({\rm div}\nu_t)|_{t=0})uw\\
&+ uw (\nabla \la_t\cdot\nu_t)|_{t=0}.\end{align*}
The proof follows by using Lemma \ref{lem}.
\end{proof}
\\
\noindent Gathering Propositions \ref{div2} and \ref{div1}, we establish the main theorem of this section, that is 
\begin{theorem}
\label{main}
The discrepancy between the scattered fields due to obstacle $D_\eps$ and the obstacle $D$ is 
\[u_\eps^s(x)- u^s(x)=-\int_{\partial D}B_\eps u(y)w(y,x)\,ds(y)+ \mathcal{O}(||\eps||^2),\]
uniformly for $x$ in some compact subset $K \subset \mathbb{R}^d \setminus \overline{D}$,
where $w(\cdot,x)$ is the solution of problem (\ref{PbInit}) associated with $u^i=\Phi(\cdot,x)$, and the surface operator $B_\eps$ is defined
by
\begin{align*}B_\eps u=(\eps\cdot\nu)(k^2-2H\la)&u+{\rm div}_{\Gamma}\left((Id+ 2\mu (R-H\,Id))(\eps\cdot\nu)\nabla_{\Gamma}u\right)+L_{\la,\mu}\left((\eps\cdot\nu)L_{\la,\mu} u\right)\\
&+(\nabla_{\Gamma}\la\cdot\eps_\Gamma)u + {\rm div}_{\Gamma}\left((\nabla_{\Gamma}\mu\cdot\eps_\Gamma)\nabla_{\Gamma}u\right),\end{align*}
with $2H:={\rm div}_{\Gamma} \nu$,  $R:=\nabla_{\Gamma}\nu$ and $L_{\la,\mu}\, \cdot:={\rm div}_{\Gamma}(\mu \nabla_{\Gamma} \,\cdot)+\la \,\cdot$ . 
\end{theorem}
\begin{proof}
From Propositions \ref{div1} and \ref{div2} it follows that
\begin{align*}&(\eps\cdot\nu)\,{\rm div}\left(-\mu_t \nabla_{\Gamma_t}u\cdot\nabla_{\Gamma_t} w \nu_t + u \nabla w + \la_t uw\nu_t\right)|_{t=0}\\
=(\eps\cdot\nu)\nabla_{\Gamma}u&\cdot\nabla_{\Gamma}w+(\eps\cdot\nu)M_\mu u\,M_\mu w-(\eps\cdot\nu)\left(k^2+\la^2-\la ({\rm div}_{\Gamma}\nu)\right)uw\\
-(\nabla_{\Gamma}\la\cdot\eps)& uw+(\nabla_{\Gamma}\mu\cdot\eps)\nabla_{\Gamma}u\cdot\nabla_{\Gamma}w
-\mu ({\rm div}_{\Gamma}\nu)(\eps\cdot\nu)\nabla_{\Gamma}u\cdot\nabla_{\Gamma}w\\
- \mu&(\eps\cdot\nu)\nabla_{\Gamma}(\nabla u\cdot\nu)\cdot\nabla_{\Gamma}w-\mu(\eps\cdot\nu)\nabla_{\Gamma}u\cdot\nabla_{\Gamma}(\nabla w\cdot\nu)\\
-\mu &(\nabla u\cdot\nu)\nabla_{\Gamma}(\eps\cdot\nu)\cdot\nabla_{\Gamma}w - \mu (\nabla w\cdot\nu)\nabla_{\Gamma}(\eps\cdot\nu)\cdot\nabla_{\Gamma}u\\
&+2\mu(\eps\cdot\nu)(\nabla_{\Gamma}u\cdot\nabla_{\Gamma}\nu\cdot\nabla_{\Gamma}w).\end{align*}
Using the boundary condition for $u$ and $w(\cdot,x)$ on $\partial D$, we obtain
\begin{align*}&(\eps\cdot\nu)\,{\rm div}\left(-\mu_t \nabla_{\Gamma_t}u\cdot\nabla_{\Gamma_t} w \nu_t + u \nabla w + \la_t uw\nu_t\right)|_{t=0}\\
=-(\eps\cdot\nu)&(k^2-2H \la) uw + (\eps\cdot\nu)\nabla_{\Gamma}u\cdot\left(Id +2\mu (R-H\,Id)\right)\cdot\nabla_{\Gamma}w\\
&-(\nabla_{\Gamma}\la\cdot\eps) uw+(\nabla_{\Gamma}\mu\cdot\eps)\nabla_{\Gamma}u\cdot\nabla_{\Gamma}w\\
&+(\eps\cdot\nu)M_\mu u\,M_\mu w-(\eps\cdot\nu)\la^2 uw\\
+ \mu(\eps\cdot\nu)&\nabla_{\Gamma}(M_\mu u+\la u)\cdot\nabla_{\Gamma}w+\mu(\eps\cdot\nu)\nabla_{\Gamma}u\cdot\nabla_{\Gamma}(M_\mu w+\la w)\\
+\mu (M_\mu u&+\la u)\nabla_{\Gamma}(\eps\cdot\nu)\cdot\nabla_{\Gamma}w + \mu (M_\mu w+\la w)\nabla_{\Gamma}(\eps\cdot\nu)\cdot\nabla_{\Gamma}u.\end{align*}
The three last lines of the above expression can be written as
\begin{align*}(\eps\cdot\nu)&(L_{\la,\mu} u)(L_{\la,\mu} w)-\la(\eps\cdot\nu)w(L_{\la,\mu}u)-\la(\eps\cdot\nu)u(L_{\la,\mu}w)\\
&+\mu\nabla_{\Gamma}((\eps\cdot\nu)L_{\la,\mu}u)\cdot\nabla_{\Gamma}w+\mu\nabla_{\Gamma}((\eps\cdot\nu)L_{\la,\mu}w)\cdot\nabla_{\Gamma}u.\end{align*}
The integral over $\partial D$ of the above expression is, after integration by parts and simplification,
\[-\int_{\partial D}(\eps\cdot\nu)(L_{\la,\mu}u)(L_{\la,\mu}w)\,ds.\]
To complete the proof, we simply use Lemma \ref{int3} and integration by parts. 
\end{proof}
\noindent
\begin{corollary}
\label{coldiff}
We assume that $\partial D$, $\la$ and $\mu$ are analytic, and $(\la,\mu)$ satisfy assumption \ref{AsPbdirect}.
Then the far field operator 
$T : \quad (\lambda,\mu,\partial{D}) \rightarrow u^{\infty}$
is differentiable with respect to $\partial D$ according to definition \ref{frechet} and its Fr\'echet derivative is given by
\[T'_{\lambda,\mu}(\partial D)\cdot\eps=v_\eps^\infty,\]
where $v_\eps^\infty$ is the
far field associated with the outgoing solution $v_\eps^s$ of the Helmholtz equation outside $D$ which satisfies the GIBC condition
\[\deri{v_\eps^s}{\nu}+{\rm div}(\mu \nabla_{\Gamma}v_\eps^s)+\la v_\eps^s=B_{\eps}u \quad \text{on} \; \partial D,\]
where $B_\eps u$ is given by Theorem \ref{main}.
\end{corollary}
\begin{proof}
Proceeding as in \cite{haddar_kress}, we can drop the assumption $\overline{D} \subset D_\eps$ provided we assume that $\partial D$, $\la$ and $\mu$ be analytic.  
The result then follows from Theorem \ref{main} and an integral representation for the scattered field $v_\eps^s=u^s_\eps-u^s$.
\end{proof}
\begin{remark}
With classical impedance boundary condition, that is $\mu=0$, we retrieve the result of \cite[theorem 2.5]{haddar_kress}.
Let us also remark that in this case the surface operator $B_\eps$ in Theorem \ref{main} is a second-order operator, while it becomes a fourth-order surface operator
when $\mu \neq 0$.
\end{remark}
\begin{remark}
\label{rem1}
Classically, the shape derivative only involves the normal part $(\eps \cdot
\nu)$ of field $\eps$ (see for example \cite[Proposition 5.9.1]{HenPie}). In
view of Theorem \ref{main}, the expression of $B_\eps$ may be split in two
parts: a first part involving only the normal component $(\eps\cdot\nu)$ and a
second part involving only the tangential component 
$\eps_\Gamma$. The presence of this second part is due to the fact that the
impedances $\la$ and $\mu$ are surface functions that depend on $\partial D$!  
\end{remark}
\section{An optimization technique to solve the inverse problem}
\label{optim}
This section is dedicated to the effective reconstruction of both the obstacle $\partial D$ and the functional impedances $(\la,\mu)$ from the observed far fields $u^\infty_{{\rm obs},j}:= T_j(\la_0,\mu_0,\partial D_0) \in L^2(S^{d-1})$ associated with $N$ given plane wave directions, where $T_j$ refers to incident direction $\hat{d}_j$. We shall minimize the cost function
\begin{equation}
\label{costfunction}
 F(\la,\mu,\partial D) = \frac{1}{2} \sum_{j=1}^N \Vert T_j(\la,\mu,\partial D) - u^{\infty}_{{\rm obs},j} \Vert^2_{L^2(S^{d-1})}
\end{equation}
with respect to $\partial D$ and $(\la,\mu)$ by using a steepest descent method. \\
To do so, we first compute the Fr\'echet derivative of $T$ with respect to $(\la,\mu)$ for fixed $D$. We have the following theorem.
\begin{theorem}
\label{thdiff}
We assume that $D$ is Lipschitz continuous. Then for $(\la,\mu) \in (L^{\infty}(\partial D))^2$ which satisfy assumption \ref{AsPbdirect}, the function $T : (\lambda,\mu,\partial{D}) \rightarrow u^{\infty}$
is Fr\'echet differentiable with respect to $(\la,\mu)$ and its Fr\'echet derivative is given by
\[T'_{\partial D}(\la,\mu)\cdot(h,l)= v_{h,l}^\infty(\hat{x}):=\left< p(\cdot,\hat{x}),{\rm div}_{\Gamma}(l\nabla_{\Gamma}u)+h\,u\right>_{H^1(\Gamma),H^{-1}(\Gamma)},\quad \forall \hat{x} \in S^{d-1},\]
where
\begin{itemize}
\item $u$ is the solution of the problem (\ref{PbInit}),
 \item $p(\cdot,\hat{x})=\Phi^\infty(\cdot,\hat{x})+p^s(\cdot,\hat{x})$ is the solution of (\ref{PbInit}) with $u^i$ replaced by
$\Phi^\infty(\cdot,\hat{x})$.
 \end{itemize}
\end{theorem}
\begin{proof}
The proof of this result can be found in \cite{BoChHa11}.
\end{proof}
\\
\noindent The Fr\'echet derivative of $T$ with respect to $\partial D$ for fixed $(\la,\mu)$ is given by Theorem \ref{main} and its corollary \ref{coldiff}.
With the help of corollary \ref{coldiff} and Theorem \ref{thdiff}, and in the case $\partial D$, $\la$ and $\mu$ are analytic,
we obtain the following expressions for the partial derivatives of the cost function $F$ with respect to $(\la,\mu)$ and $\partial D$ respectively.
\begin{align} 
F'_{\partial D}(\la,\mu) &\cdot (h,l)  =
\sum_{j=1}^N \Re e\left(\int_{\partial D}G_j({\rm div}_{\Gamma}(l\nabla_{\Gamma}u_j)+ hu_j)\,dy \right),
\label{dlamu}\\
&F'_{\la,\mu}(\partial D)\cdot \eps  =-\sum_{j=1}^N \Re e \left(\int_{\partial D} G_j(B_\eps u_j)\,dy\right)
\label{dobs}
\end{align}
where 
\begin{itemize}
\item $u_j$ is the solution of the problem (\ref{PbInit}) which is associated to plane wave direction $\hat{d}_j$,
\item $G_j=G_j^i+G_j^s$ is the  solution of problem (\ref{PbInit}) with $u^i$ replaced by
\[
G_j^i(y) := \int_{S^{d-1}}\Phi^{\infty}(y,\hat{x}) \overline{(T_j(\la,\mu,\partial D) -u^{\infty}_{{\rm obs},j})}\,d\hat{x}.
\]
\end{itemize}
In the numerical part of the paper we restrict ourselves to the two dimensional setting, that is $d=2$.
The minimization of the cost function $F$
alternatively with respect to $D$, $\la$ and $\mu$ relies on the directions of steepest descent given by (\ref{dlamu}) and (\ref{dobs}).
The minimization with respect to $(\la,\mu)$ is already exposed in \cite{BoChHa11}, so that we only describe the minimization with respect to $D$.
It is essential to remark from Theorem \ref{main} that the partial derivative with respect to $D$ depends only on the values of $\eps$ on $\partial D$.
With the decomposition $\eps =\eps_\tau \tau + \eps_\nu \nu$, where $\tau$ is the tangential unit vector, we formally compute $(\eps_\tau,\eps_\nu)$ on $\partial D$ such that
\[\eps_\tau \tau + \eps_\nu \nu= -\alp F'_{\la,\mu}(\partial D),\]
where $\alp>0$ is the descent coefficient.
In order to decrease the oscillations of the updated boundary, similarly to \cite{BoChHa11} we use a $H^1$-regularization, that is we search $\eps_\tau$ and $\eps_\nu$ in $H^1(\partial D)$ such that
for all $\phi \in H^1(\partial D)$,
\begin{align} &\eta_\tau \int_{\partial D}\nabla_\Gamma \eps_\tau\cdot\nabla_\Gamma \phi\,ds + \int_{\partial D}\eps_\tau \phi\,ds=-\alp F'_{\la,\mu}(\partial D)\cdot(\phi\, \tau),\label{pbtau}\\
 &\eta_\nu \int_{\partial D}\nabla_\Gamma \eps_\nu\cdot\nabla_\Gamma \phi\,ds + \int_{\partial D}\eps_\nu \phi\,ds=-\alp F'_{\la,\mu}(\partial D)\cdot(\phi\, \nu),\label{pbnu}\end{align}
where $\eta_\tau$, $\eta_\nu>0$ are regularization coefficients, while $F'_{\la,\mu}(\partial D)\cdot(\phi\,\tau)$ and $F'_{\la,\mu}(\partial D)\cdot(\phi\,\nu)$ are given by (\ref{dobs}) (see more explicit expressions in \cite{BoChHa11report}).
The updated obstacle $D_\eps$ is then obtained by moving the mesh points $x$ of $\partial D$ to the points $x_\eps$ defined by
$x_\eps=x + (\eps_\tau \tau + \eps_\nu \nu)(x)$, while the extended impedances on $\partial D_\eps$ are defined, following (\ref{transport}), by $\la_\eps(x_\eps)=\la(x)$ and $\mu_\eps(x_\eps)=\mu(x)$.
The points $x_\eps$ enable us to define a new domain $D_\eps$, and we have to remesh the complementary domain $\Omega^\eps_R=B_R \setminus \overline{D_\eps}$ to solve the next forward problems.
The descent coefficient $\alp$ and the regularization parameters $\eta_\tau,\eta_\nu$ are determined as follows: $\alp$ is increased (resp. decreased) and $\eta_\tau,\eta_\nu$ are decreased (resp. increased) as soon as the cost function decreases (resp. increases). The algorithm stops as soon as $\alp$ is too small. 
With the help of the relative cost function, namely
\begin{equation}
\label{relative_costfunction}
 {\rm Error}: = \frac{1}{N} \sum_{j=1}^N \frac{\Vert T_j(\la,\mu,\partial D) - u^{\infty}_{{\rm obs},j} \Vert_{L^2(S^1)}}{\Vert u^{\infty}_{{\rm obs},j} \Vert_{L^2(S^1)}},
\end{equation}
we are able to determine if the computed $(\la,\mu,\partial D)$ corresponds to a global or a local minimum: in the first case ${\rm Error}$ is approximately equal to the amplitude of noise while in the second case it is much larger.
\section{Numerical results}
In order to handle dimensionless impedances, we replace $\la$ by $k\la$ and $\mu$ by $\mu/k$ in the boundary condition of problem (\ref{PbBorne}) without changing the notations. 
Problem (\ref{PbBorne}) is solved by using a finite element method based on the
variational formulation associated with problem (\ref{PbBorne}) and which is
introduced in \cite{BoChHa10report}. We  used classical Lagrange finite elements. The variational formulations (\ref{pbtau}) (\ref{pbnu}) as well as those used to update the impedances $\la$ and $\mu$ (see \cite{BoChHa10report}) are solved by using the same finite element basis.  
All computations were performed with the help of the software FreeFem++
\cite{FF++}. We obtain some artificial data with forward computations for some
given data $(\la_0,\mu_0,\partial D_0)$. The resulting far fields $u_{{\rm
    obs},j}^\infty$, $j=1,N$ are then corrupted with Gaussian noise of various amplitudes. More precisely, for each Fourier coefficient of the far field we compute a Gaussian noise with normal distribution. Such a perturbation is multiplied by a constant which is calibrated in order to obtain a global relative $L^2$ error of prescribed amplitude: $1 \%$ or $5 \%$.\\
We use the same finite element method to compute the artificial data and to solve the forward problem during the iterations of the inverse problem.
However, we avoid the ``inverse crime" for two reasons. First, the mesh used to obtain the artificial data is different from the one used to initialize the identification process. Secondly, as said before, the artificial data are contaminated by some Gaussian noise. 
In figure \ref{Fig8} we show, on a particular example, the mesh used to compute the artificial data, the mesh based on our initial guess and some intermediate mesh 
obtained during the iterations of the inverse problem.
\begin{figure}[h!]
\begin{minipage}{\textwidth}
\subfigure[Used mesh for artificial data]{\includegraphics[width =.40\textwidth]{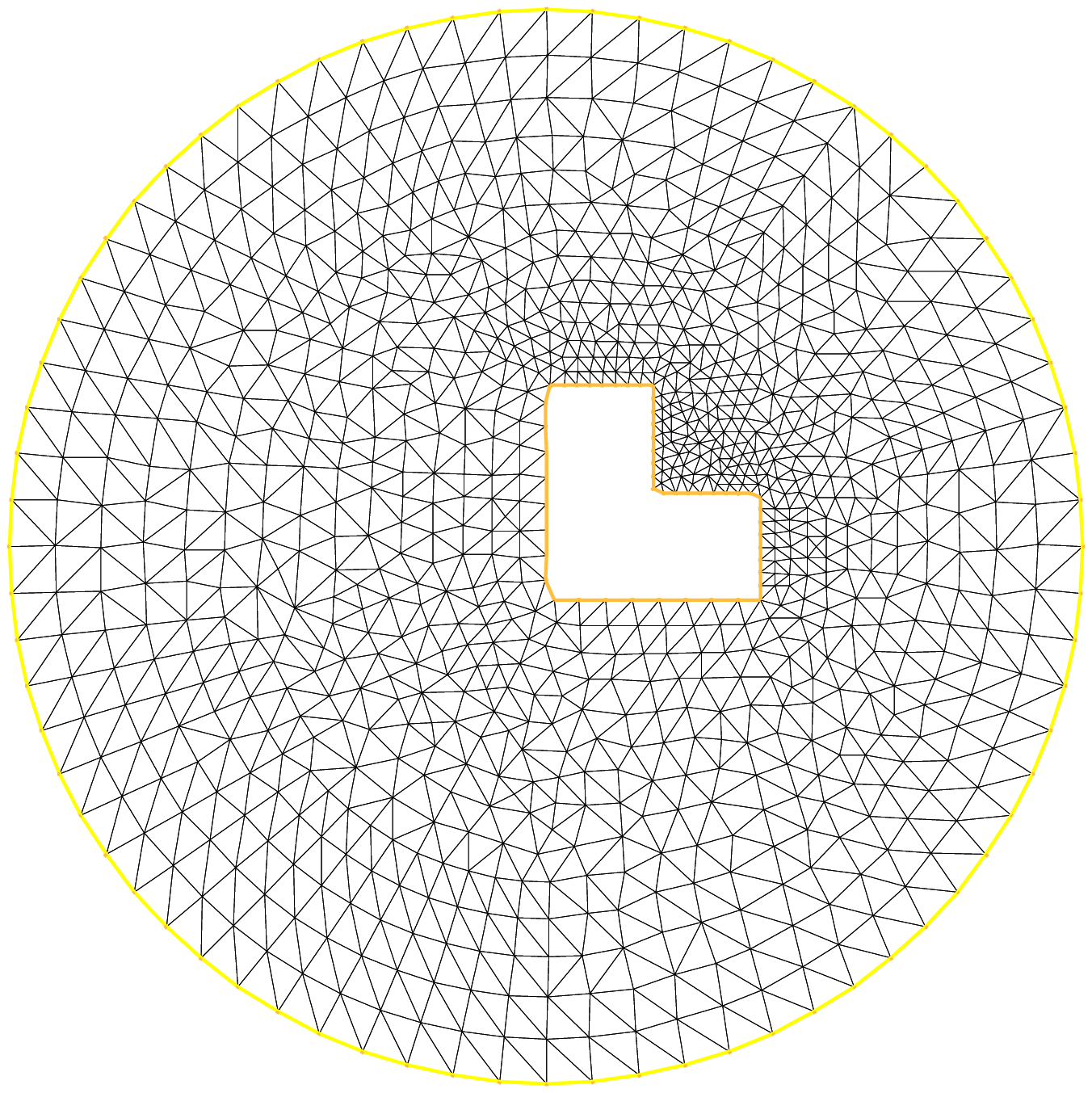}}
\subfigure[Mesh for the initial guess]{\includegraphics[width =.285\textwidth]{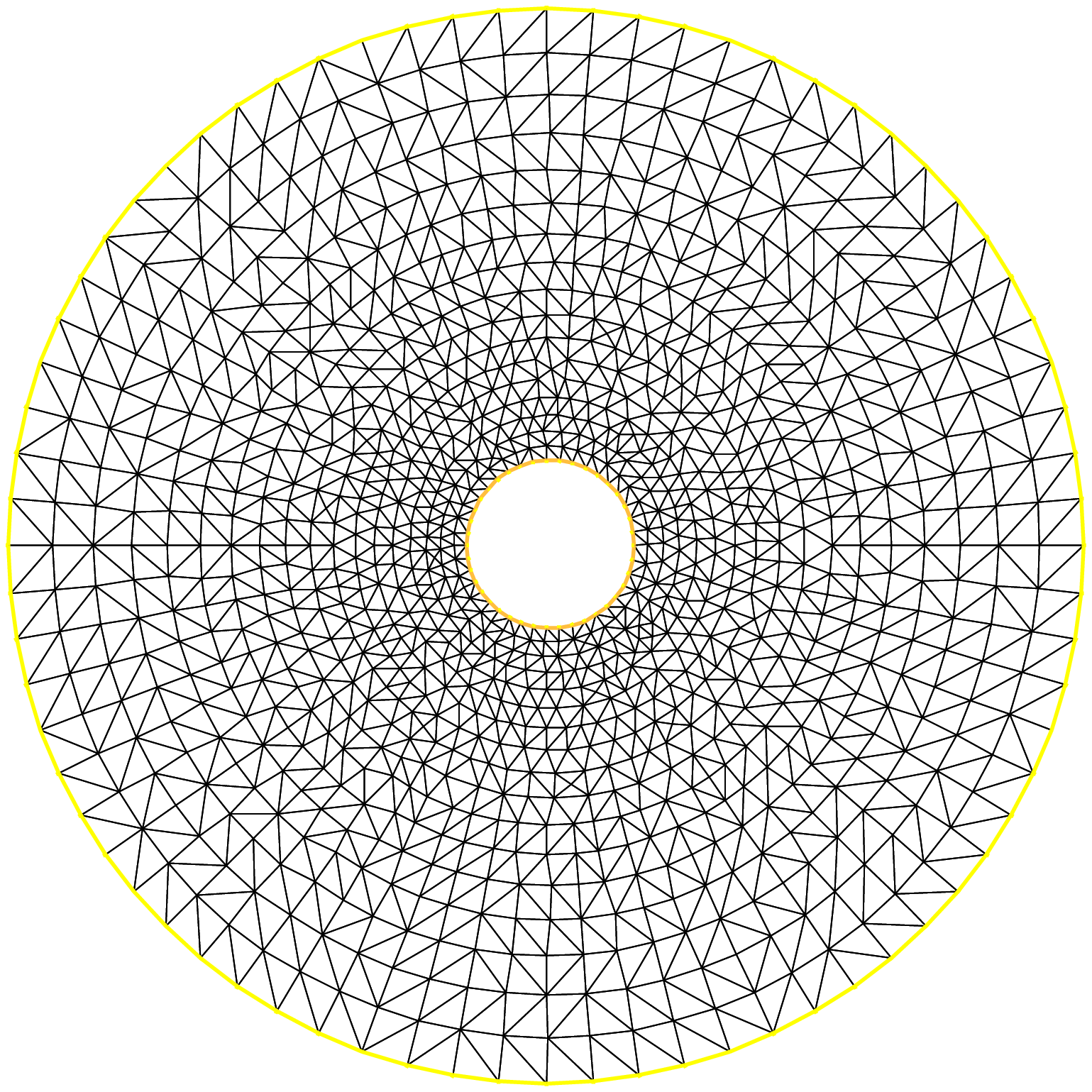}}
%\begin{center}
\subfigure[Mesh for an intermediate iteration mesh]{\includegraphics[width =.285\textwidth]{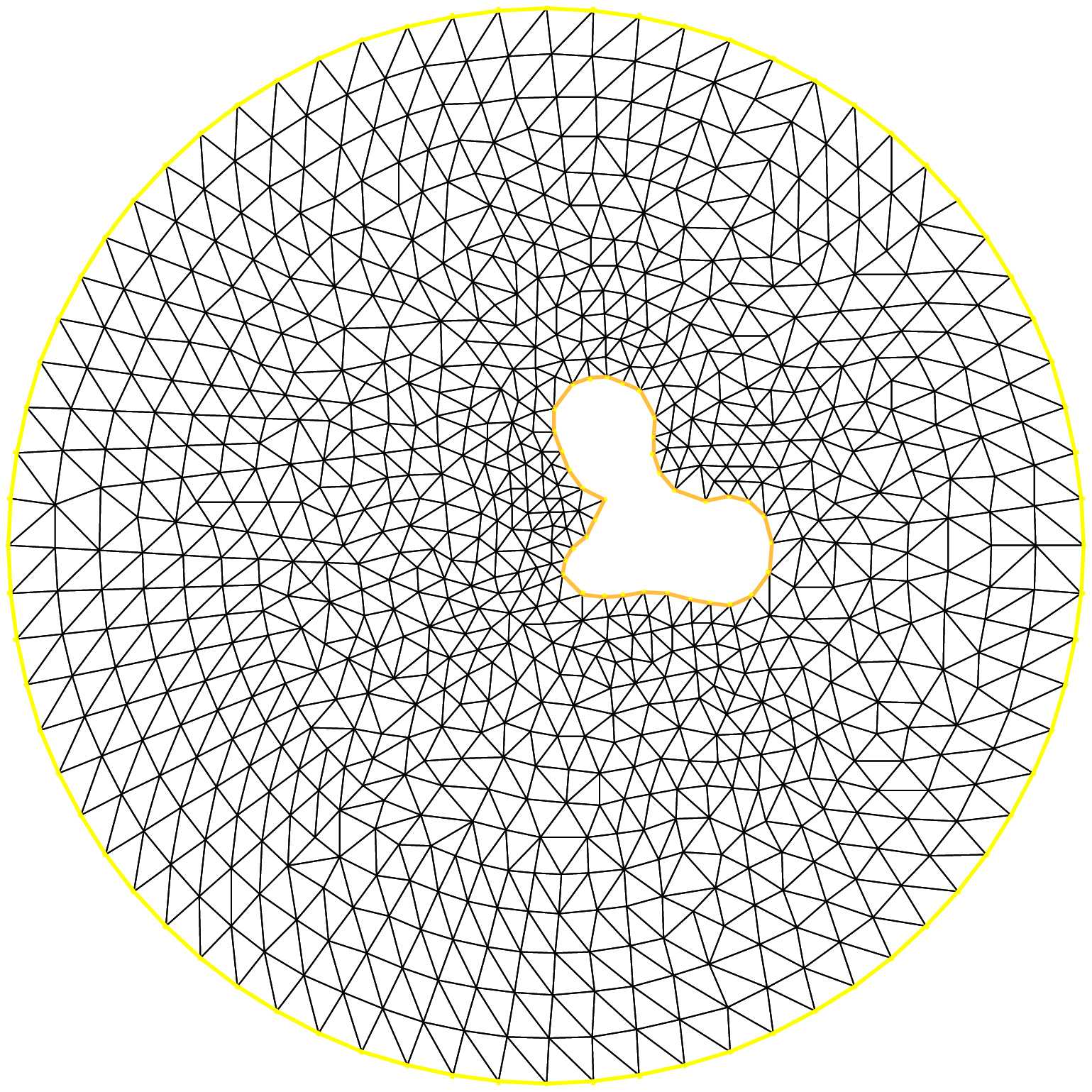}}
%\end{center}
\caption{Example of used meshes}
\label{Fig8}
\end{minipage}
\end{figure}
\subsection{Reconstruction of an obstacle with known impedances}
First we try to reconstruct an $L-$shaped obstacle $D_0$ with known impedances $(\la_0,\mu_0)$. 
The result is shown in Figure \ref{Fig2} in the case of $1\%$ noise by using
only two incident waves, i.e. $N=2$. 
The results are shown in Figure \ref{Fig3} in the case of $5\%$ noise and $N=2$, as well as $5\%$ noise and $N=8$, respectively. This enables us to test the influence of the amplitude of noise as well as 
the influence of the number of incident waves. 
In order to evaluate the impact of the initial guess on the quality of the reconstruction, we consider another initial guess which is farther from the true obstacle, in the presence of $5\%$ noise. 
We obtain poor reconstructions when only two incident waves are used (see
Figures \ref{Fig4}(a) and \ref{Fig3}(a)) while the accuracy is much better in
the case of eight incident waves (see
Figures \ref{Fig4}(b) and \ref{Fig3}(b)).
In the remainder of the numerical section all reconstructions will be done using eight incident waves.
\begin{figure}[h!]
\begin{minipage}{\textwidth}
\subfigure[Intermediate steps]{ \includegraphics[width =.4\textwidth]{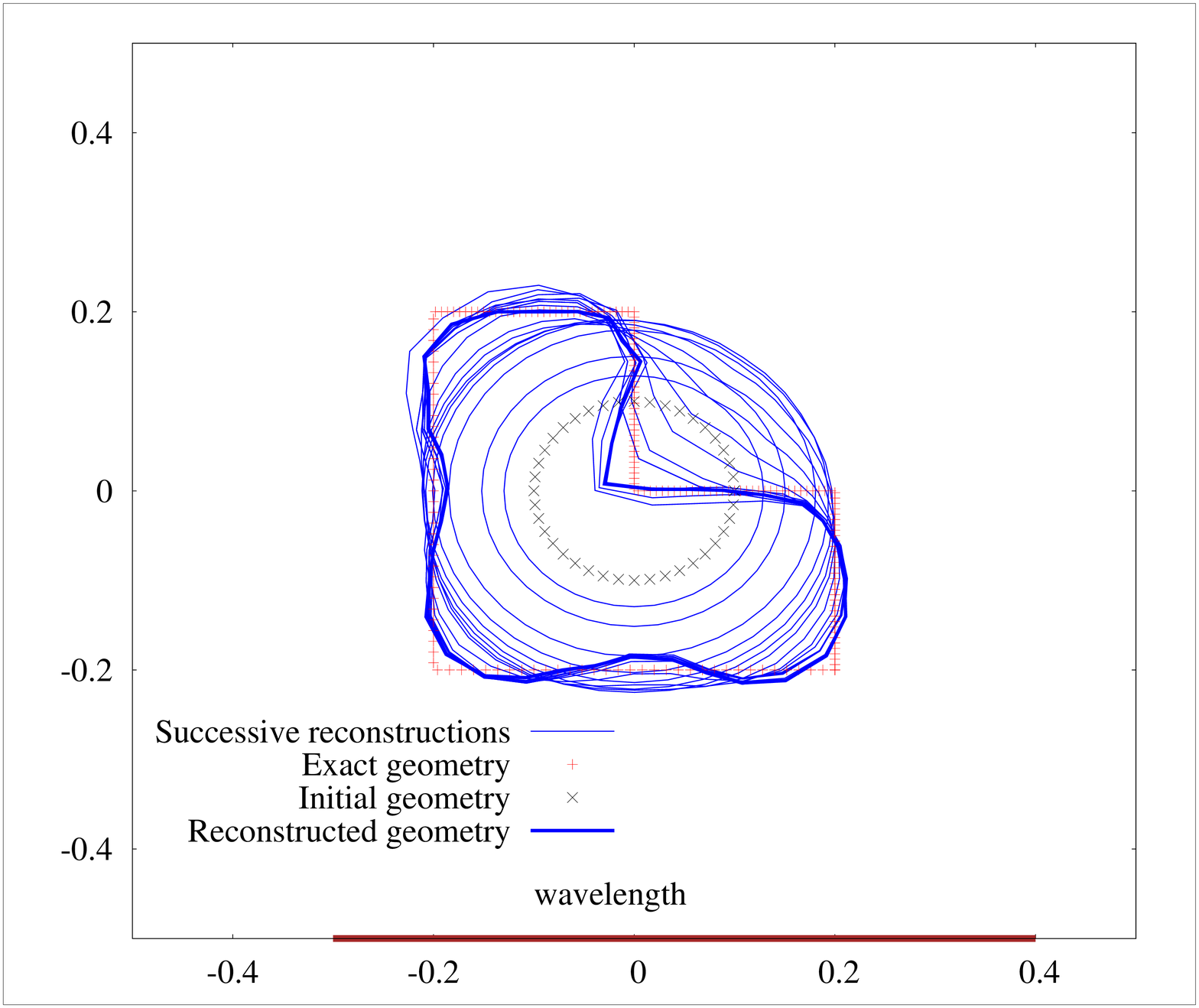}}
\hfill\subfigure[$1\%$ noise, $N=2$, Final Error = $2.1 \%$]{ \includegraphics[width =.4\textwidth]{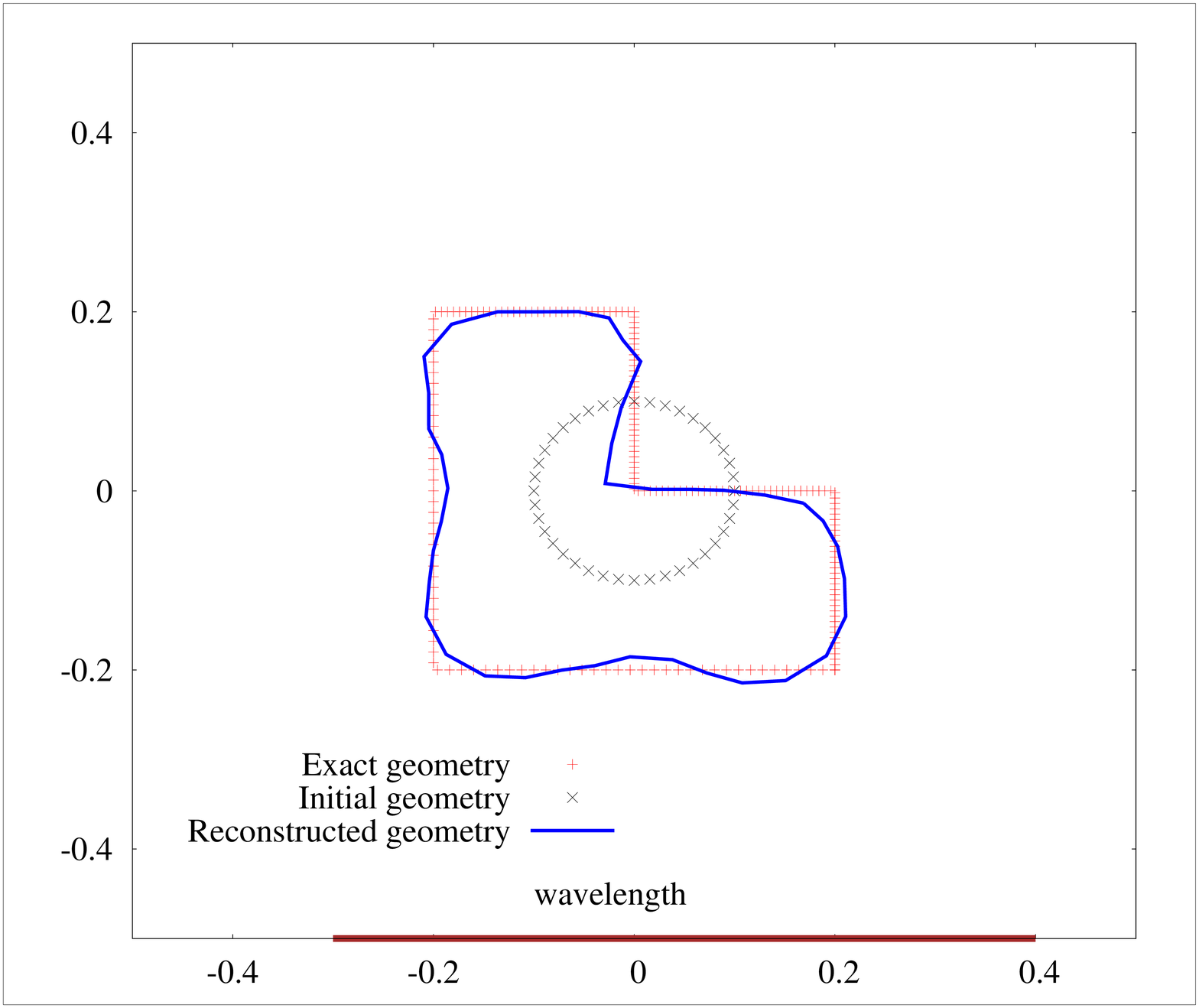}}
\caption{Case of known impedances and good initial guess}
\label{Fig2}
\end{minipage}
\end{figure}
\begin{figure}[h!]
\begin{minipage}{\textwidth}
\subfigure[$5\%$ noise, $N=2$, Final Error = $6.3 \%$]{ \includegraphics[width =.4\textwidth]{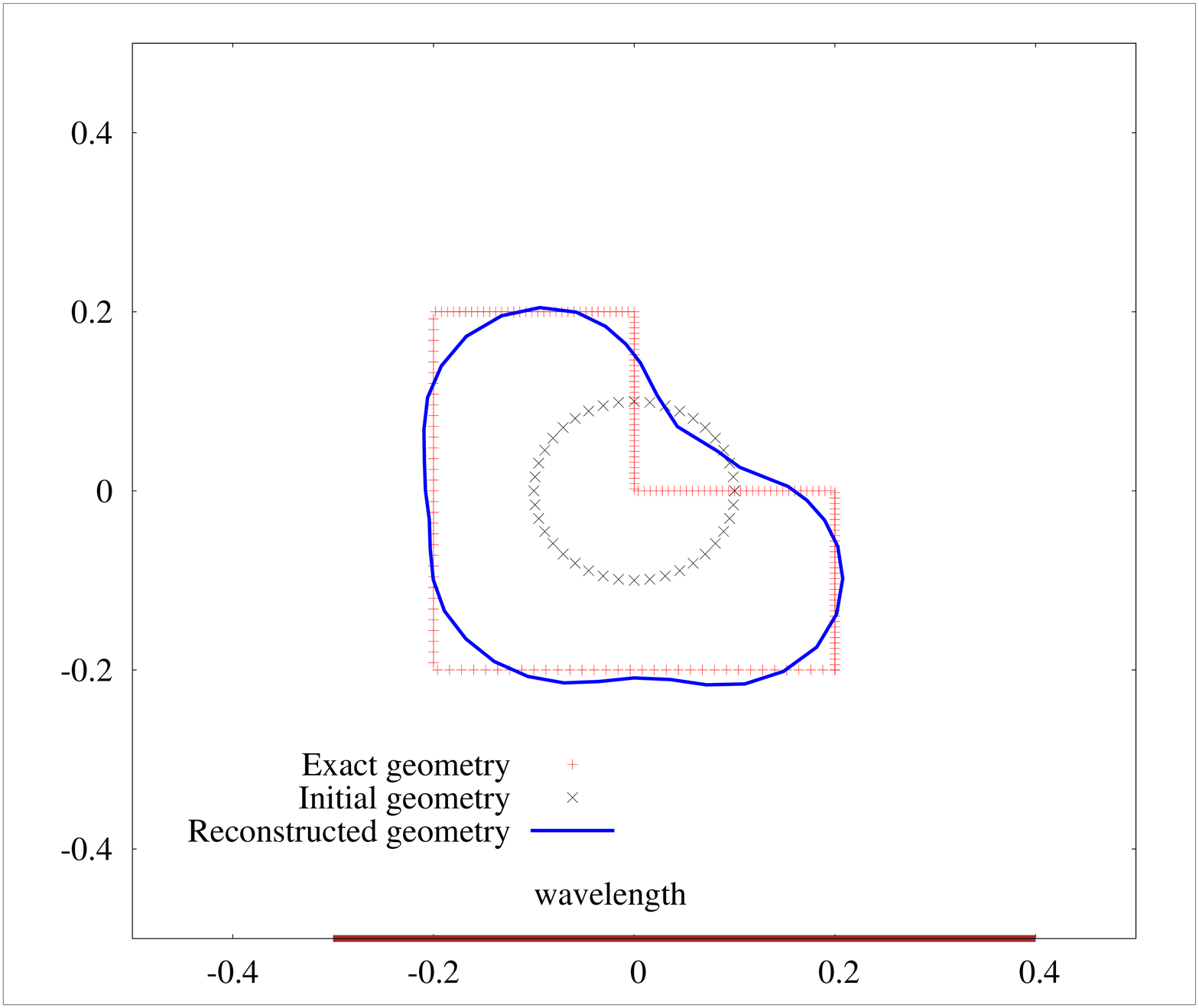}}
\hfill\subfigure[$5\%$ noise, $N=8$, Final Error = $5.0 \%$]{ \includegraphics[width =.4\textwidth]{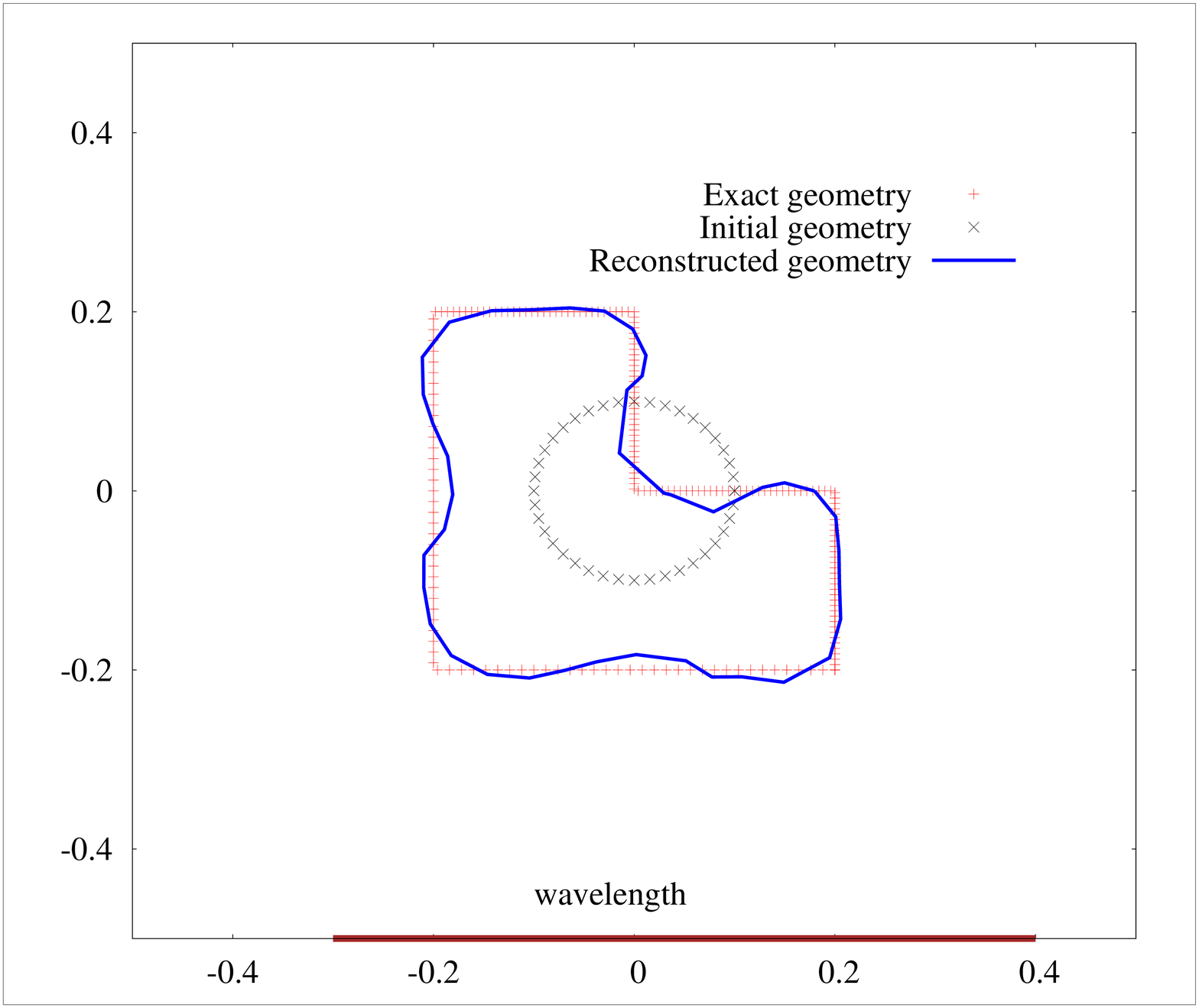}}
\caption{Case of known impedances and good initial guess, influence of noise and of $N$}
\label{Fig3}
\end{minipage}
\end{figure}
\begin{figure}[h!]
\begin{minipage}{\textwidth}
\subfigure[$5\%$ noise, $N=2$, Final Error = $25.2 \%$]{\includegraphics[width =.4\textwidth]{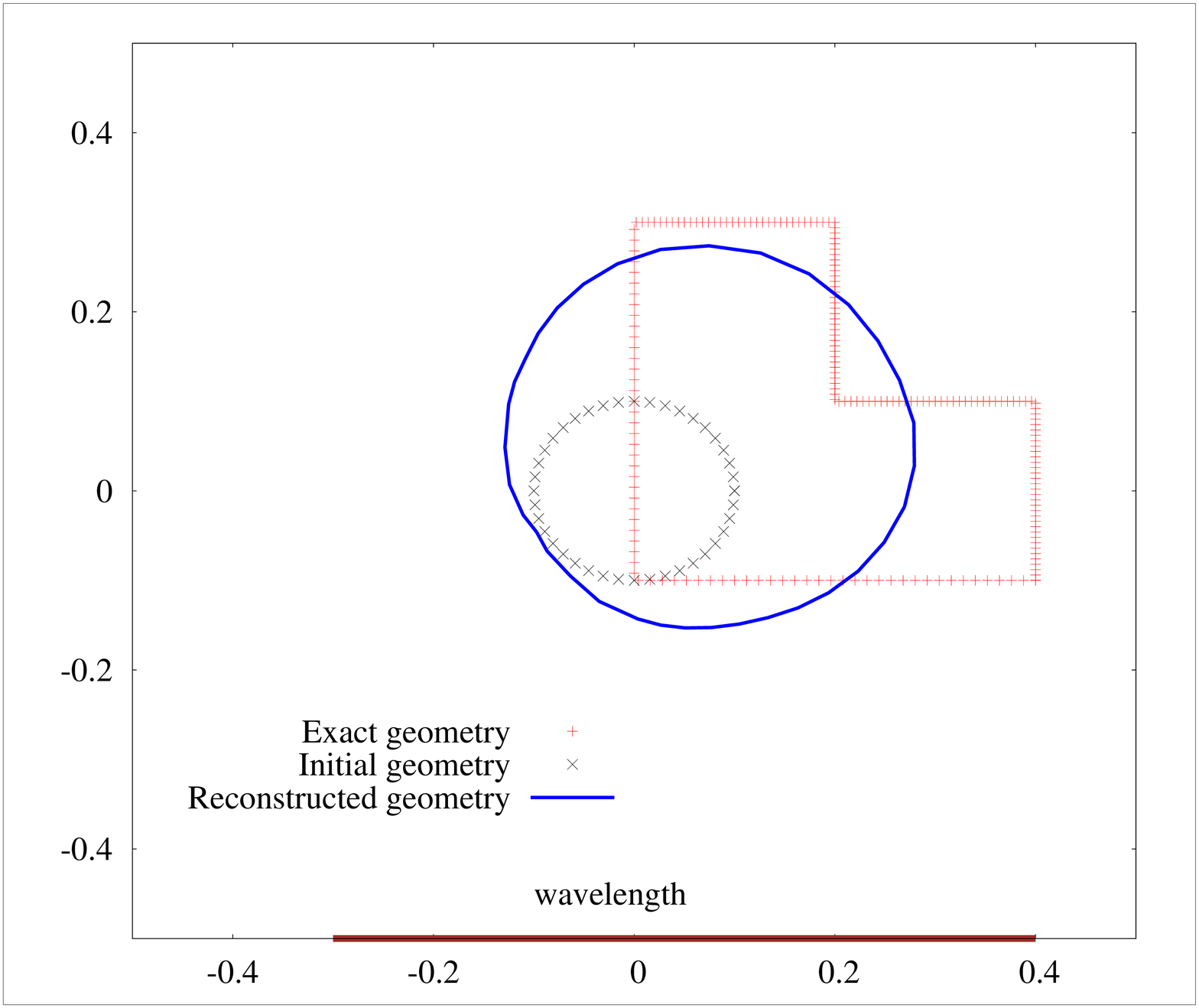}}
\hfill\subfigure[$5\%$ noise, $N=8$, Final Error = $5.8 \%$]{\includegraphics[width =.4\textwidth]{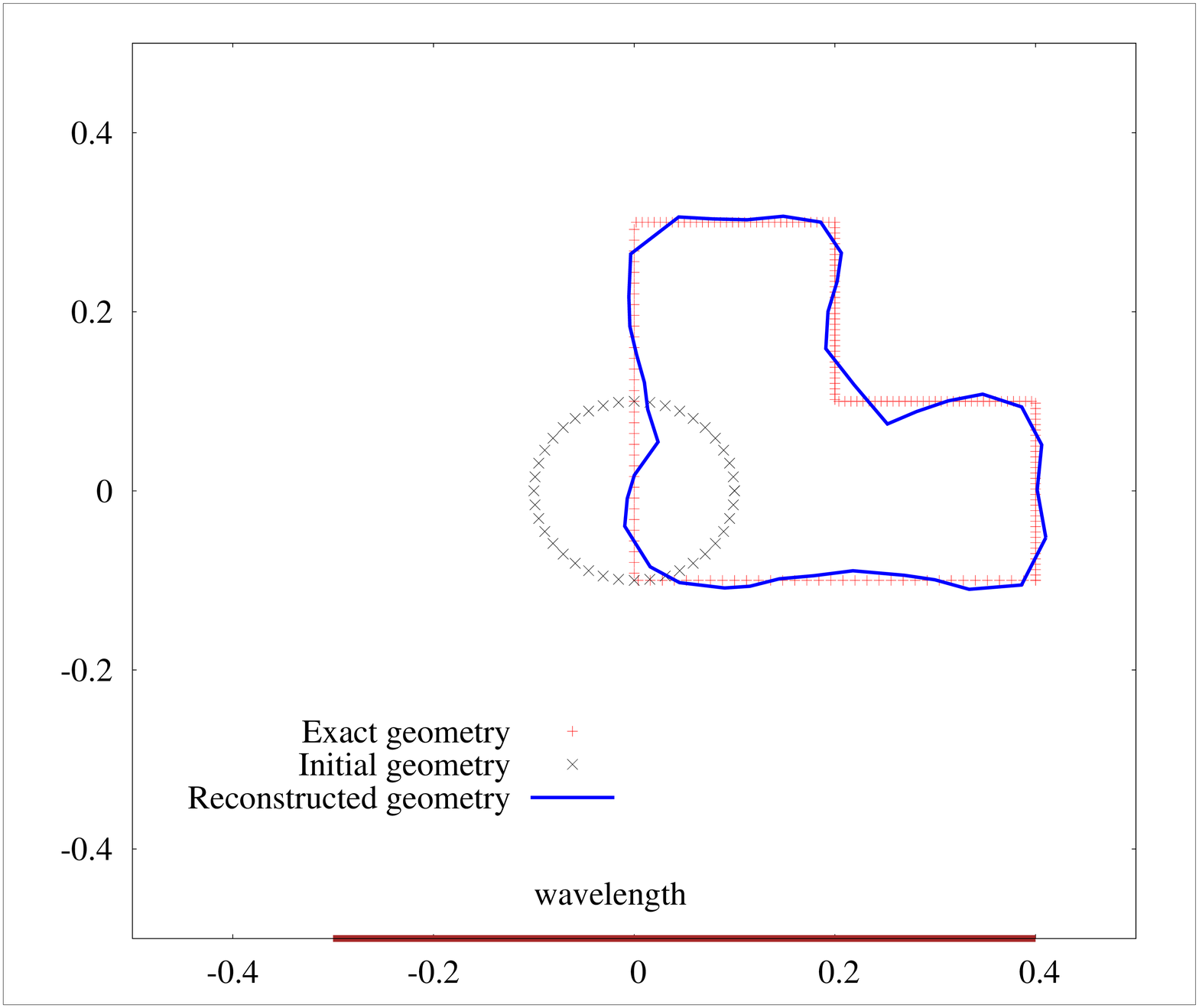}}
\caption{Case of known impedances and bad initial guess, increasing number of incident waves}
\label{Fig4}
\end{minipage}
\end{figure}
\subsection{Reconstruction of the geometry and constant impedances}
Secondly we assume that both the obstacle $D_0$ and the impedances $(\la_0,\mu_0)=(0.5i,2)$ are unknown, but these impedances are constants.
Starting from $(i,1.5)$ as initial guess for $(\la,\mu)$, the retrieved impedances are $(\la,\mu)=(0.49i,1.99)$ for $1\%$ noise and $(\la,\mu)=(0.51i,1.93)$ for $5\%$ noise, while the corresponding retrieved obstacles are shown on Figure \ref{Fig5}.
\begin{figure}[h!]
\begin{minipage}{\textwidth}
\subfigure[$1\%$ of noise, $N=8$, Final Error = $1.4 \%$]{ \includegraphics[width =.4\textwidth]{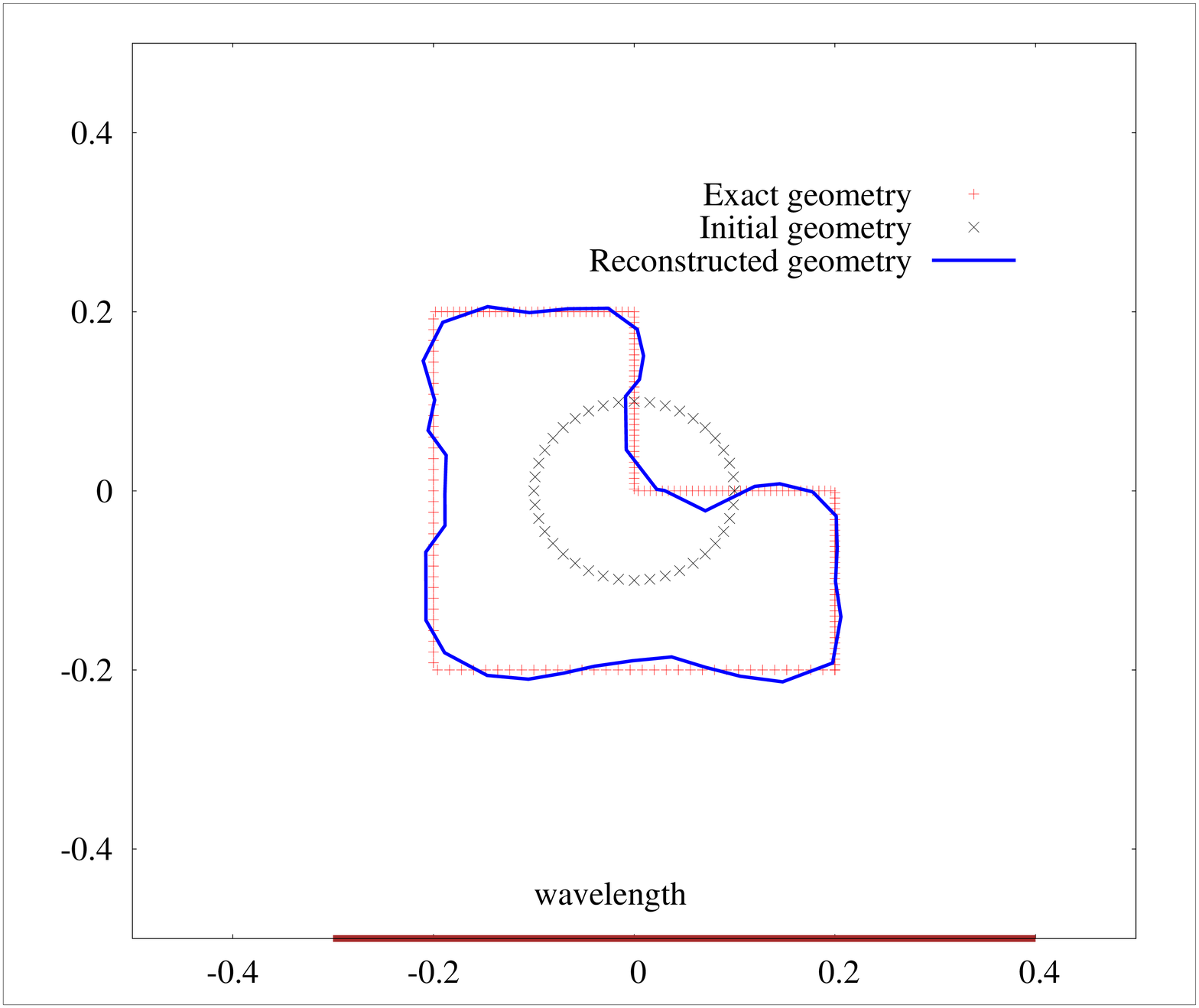}}
\hfill \subfigure[$5\%$ of noise, $N=8$, Final Error = $5.8 \%$]{ \includegraphics[width =.4\textwidth]{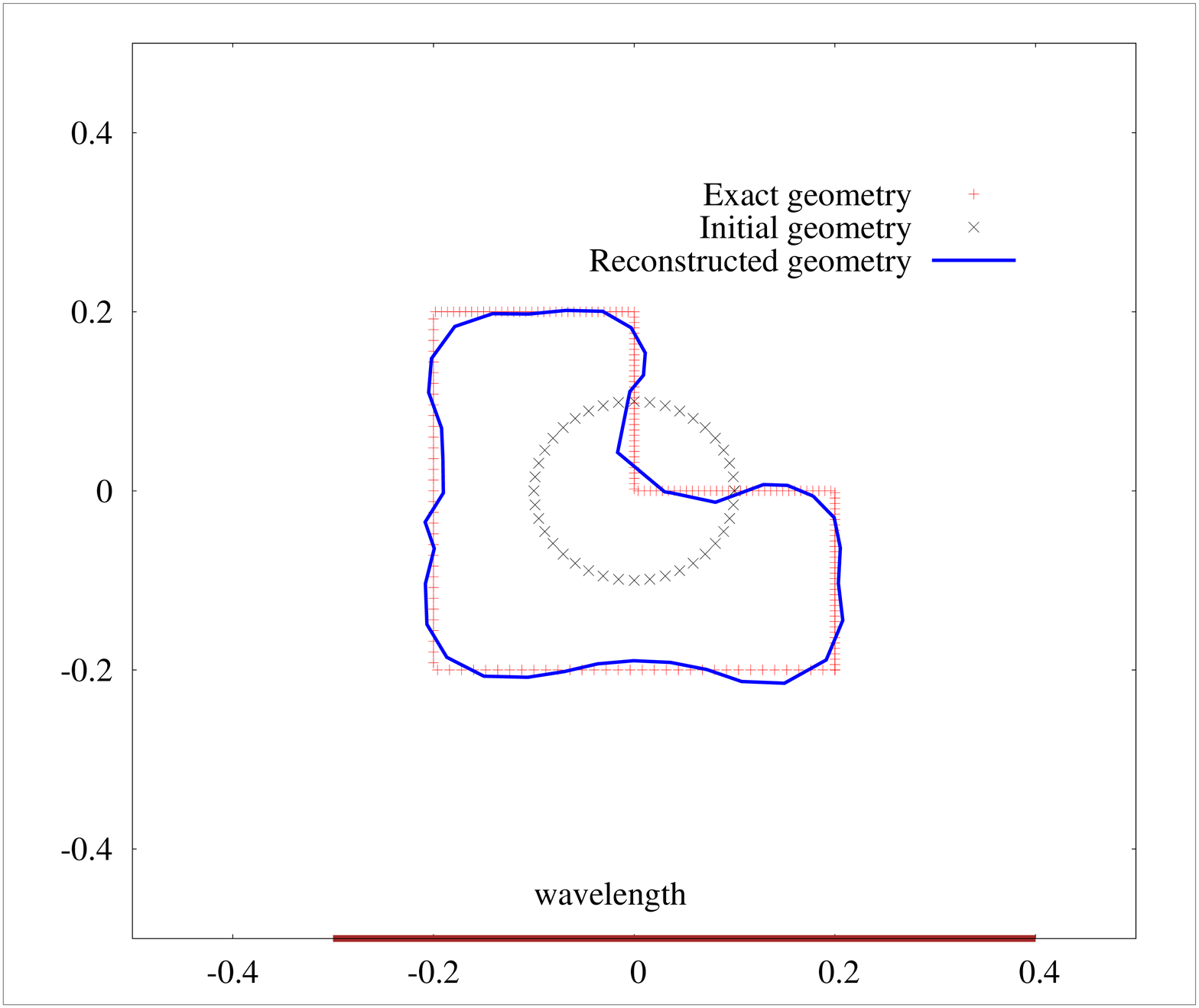}}
\caption{Case of constant but unknown impedances}
\label{Fig5}
\end{minipage}
\end{figure}
\subsection{Reconstruction of the geometry and functional impedances}
In order to emphasize the role played by the tangential part of the mapping $\eps$ in the optimization of the cost function $F$ for functional impedances (see remark \ref{rem1}), we first consider a very academic case. We try to reconstruct a circle $D_0$ of radius $R_0=0.3$ and an impedance $\la_0(\theta)=0.5(1+\sin^2(\theta+ \frac{\pi}{6}))$, where $\theta$ is the polar angle, starting from an initial circle of same center and radius $0.2$ and from the initial impedance $\la(\theta)=0.5(1+\sin^2(\theta))$. 
Compared to the true obstacle, the initial guess is hence a smaller and rotated circle.
Here $\mu=0$ for sake of simplicity. 
The amplitude of noise is $5\%$ and we use eight incident waves.
As can be seen in Figure \ref{Fig6}, the obstacle $D_0$ and the impedance $\la_0$ are quite well reconstructed even if we use only the gradient iterations on the geometry 
(we do not use the gradient iterations on the impedance).\\
We end this numerical section with a more complicated example. The goal is to
retrieve the obstacle $D_0$ defined with polar coordinates $(r,\theta)$ by
$r=0.3+0.08\cos(3\theta)$, as well as the impedances
$\la_0=0.5(1+\sin^2\theta)i$ and $\mu_0=0.5(1+\cos^2\theta)$, assuming that
both the real part of $\la$ and the imaginary part of $\mu$ are $0$, in the
presence of $5\%$ noise and using eight incident waves. Notice that in this
case the obstacle is star-shaped, which is not necessary to apply our
optimization process, but it enables us to compare the retrieved and the exact
impedances in a simple way. The results are presented in Figure \ref{Fig7} and
show fairly accurate reconstructions.
\begin{figure}[h!]
\begin{minipage}{\textwidth}
\hfill \subfigure[Reconstruction of $D$]{ \includegraphics[width =.33\textwidth]{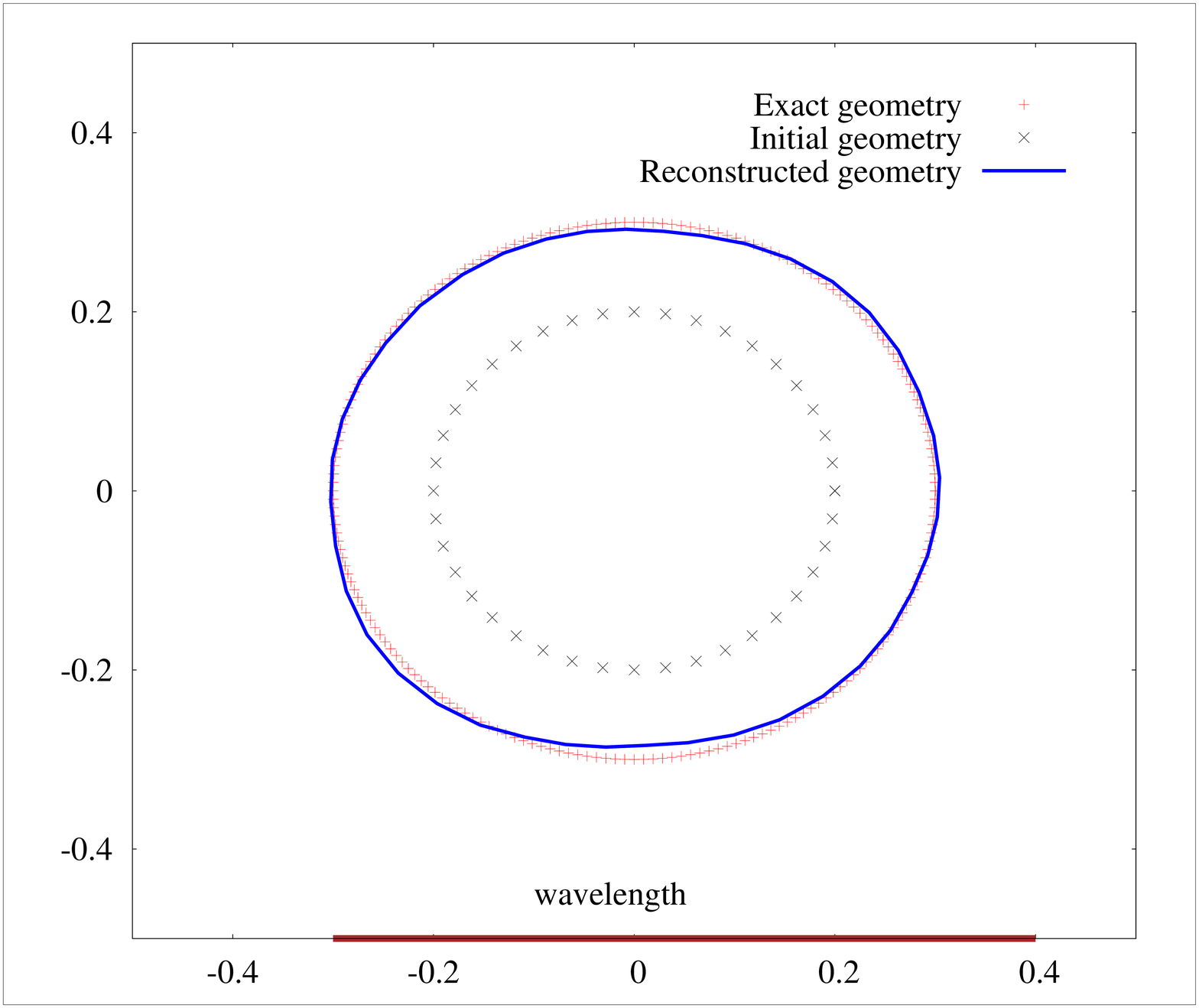}}
\hfill\subfigure[Reconstruction of ${\mathcal I}m\la$]{ \includegraphics[width =.33\textwidth]{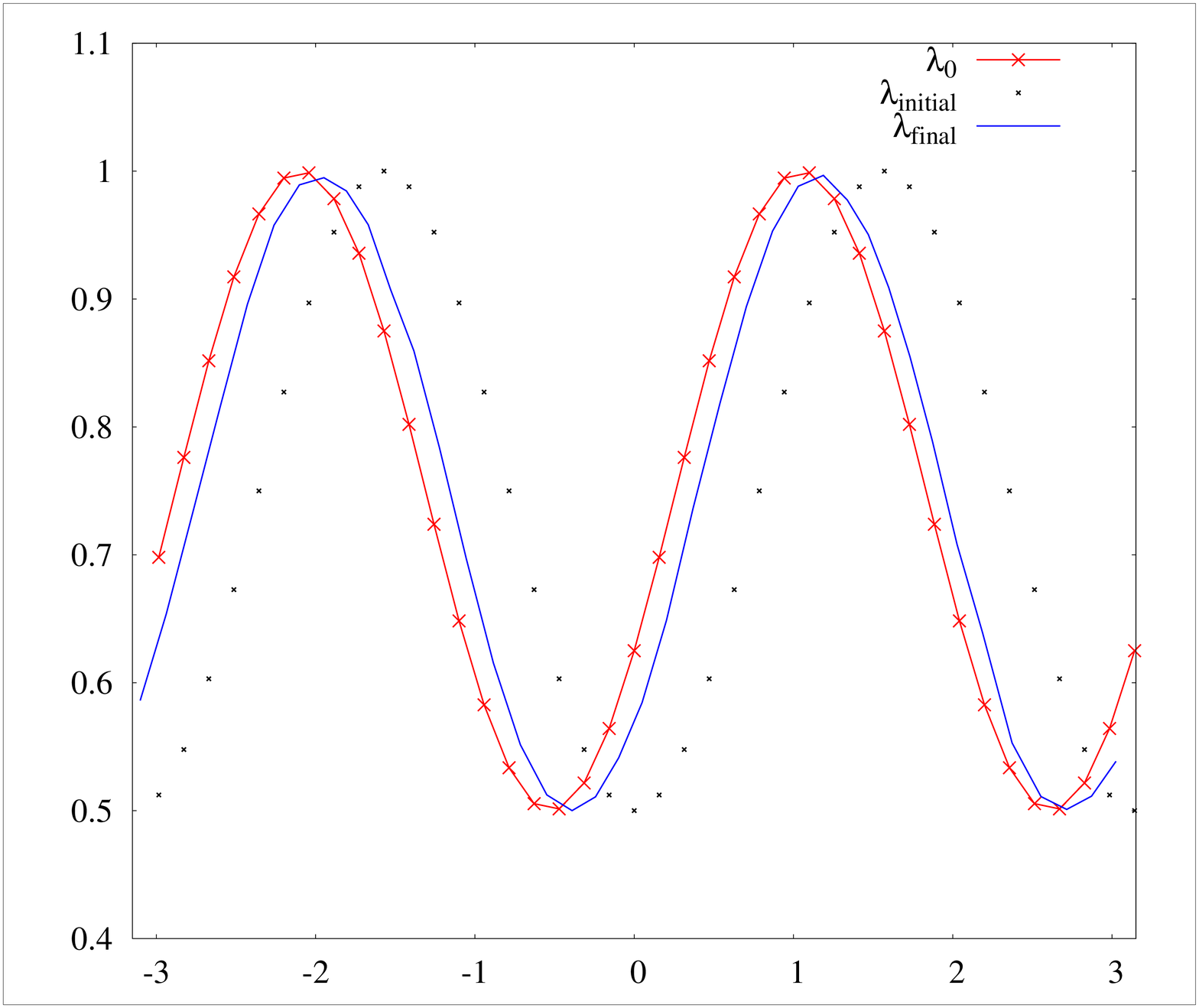}} \hfill
$\,$
\caption{Reconstruction of a circular obstacle and the impedance using only the
  derivative with respect to $D$, $N=8$}
\label{Fig6}
\end{minipage}
\end{figure}
\begin{figure}[h!]
\begin{minipage}{\textwidth}
\subfigure[Reconstruction of ${\mathcal I}m\la$]{\includegraphics[width =.325\textwidth]{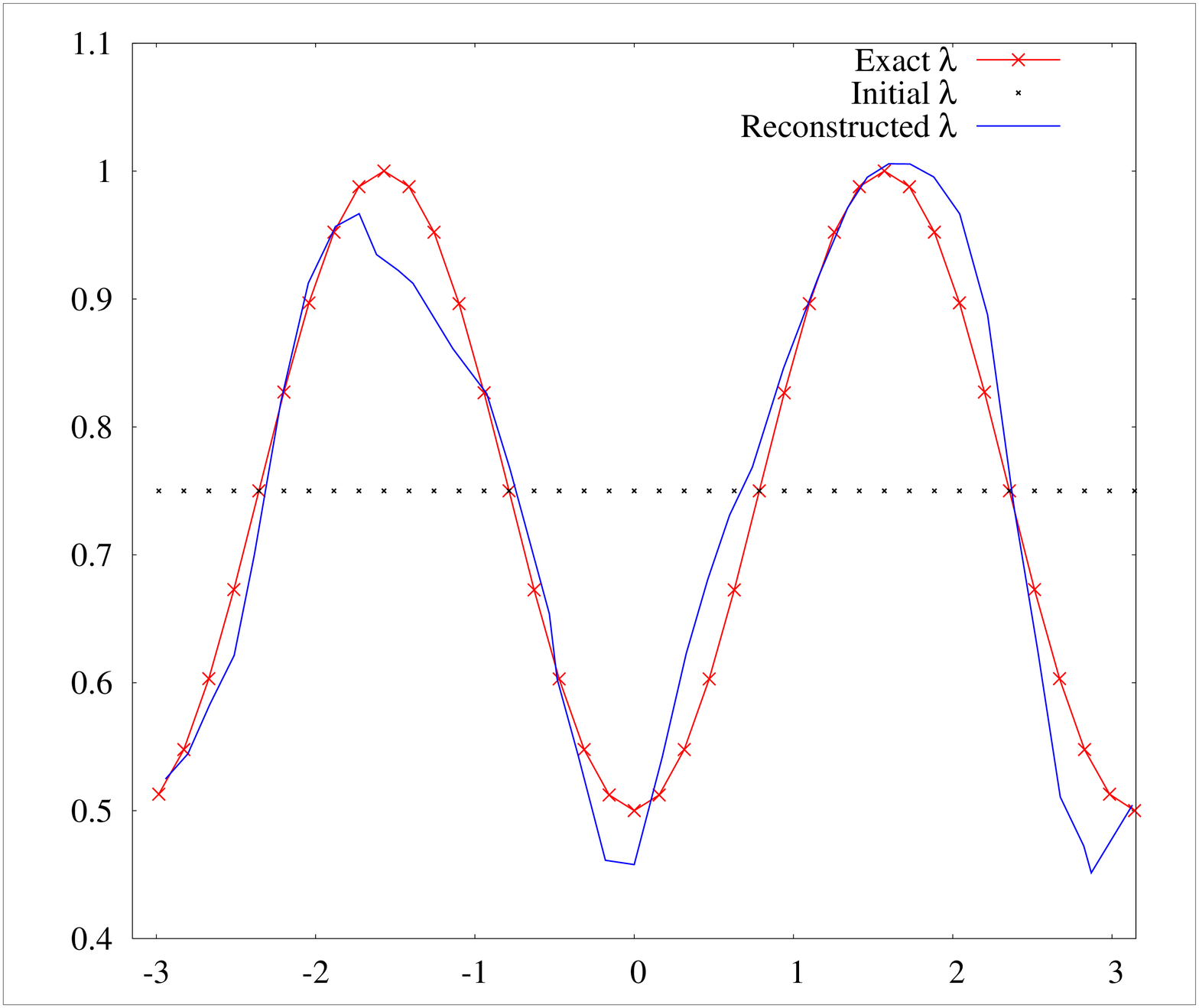}}
\subfigure[Reconstruction of ${\mathcal R}e\mu$]{\includegraphics[width =.325\textwidth]{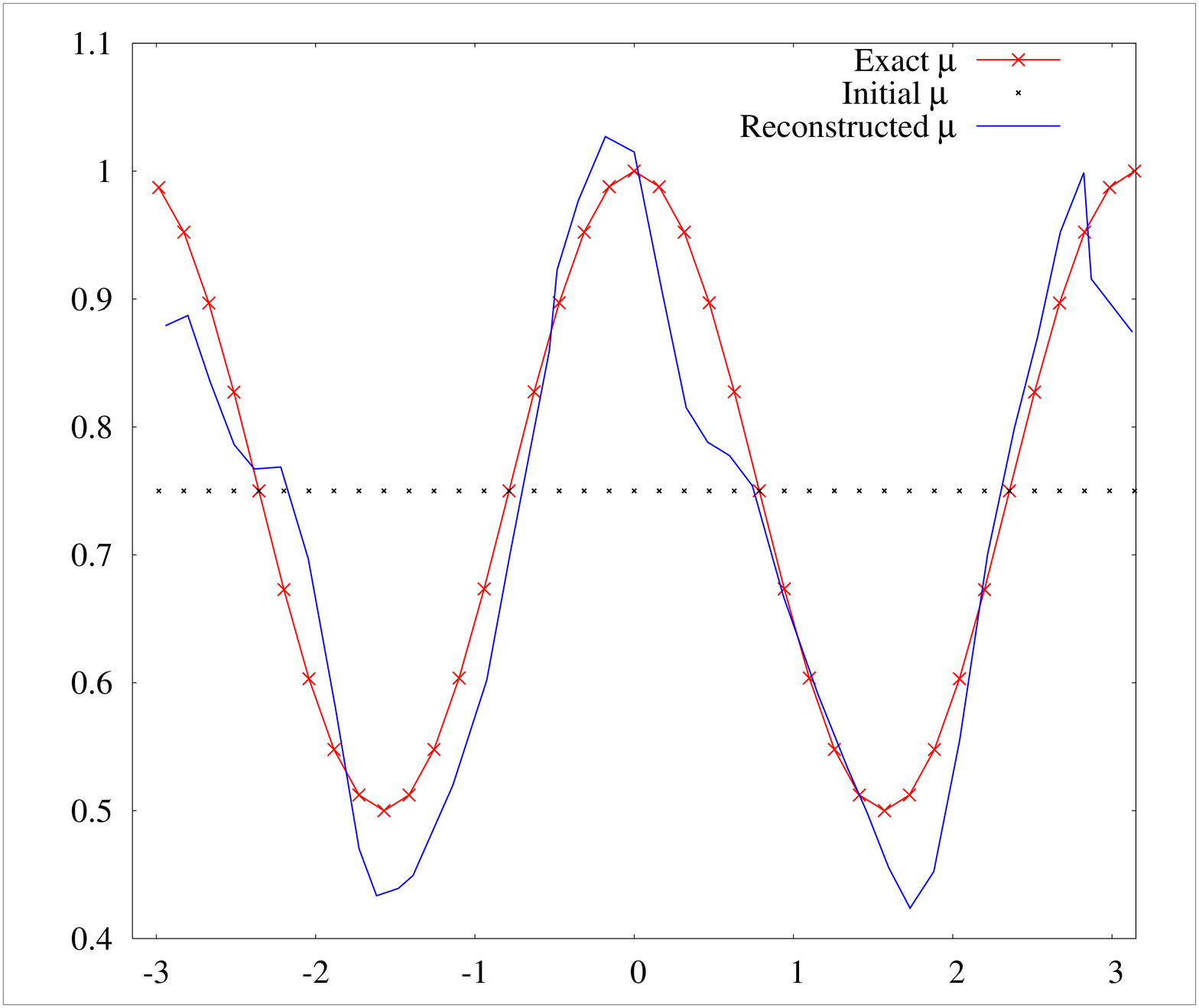}}
%\begin{center}
\subfigure[Reconstruction of $D$]{\includegraphics[width =.325\textwidth]{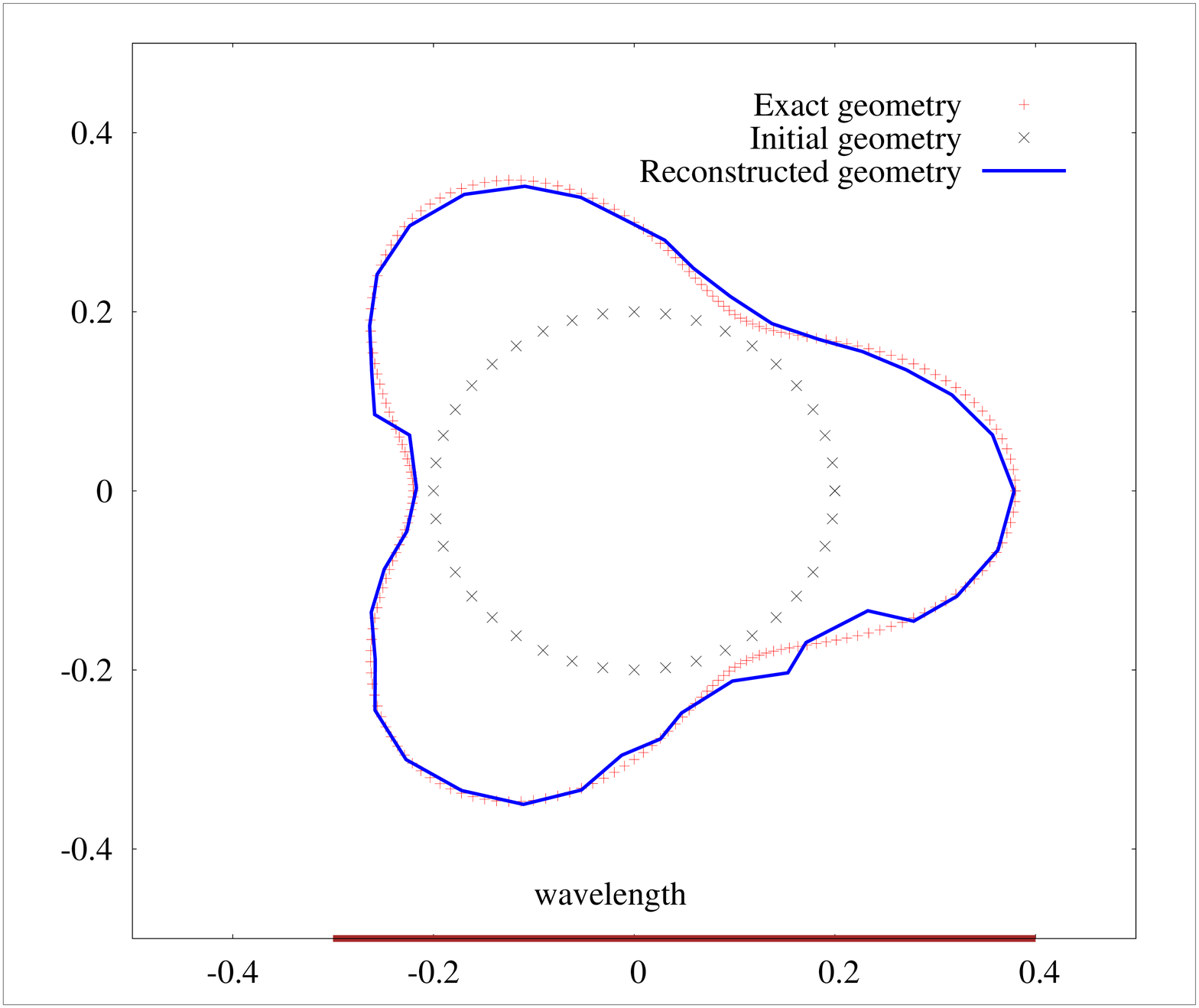}}
%\end{center}
\caption{Case of unknown arbitrary geometry and unknown variable impedances, $N=8$}
\label{Fig7}
\end{minipage}
\end{figure}
\begin{remark}
It should be noticed that our numerical results are somewhat irregular, both for the obstacle and for the impedances.
Concerning the obstacle, this is due to the fact that we update the boundary $\partial D$ of the obstacle point by point following the procedure described at the end of section \ref{optim}. Concerning the impedances, this is due to the fact that the discretization basis for $(\lambda,\mu)$ is formed by the traces of the finite elements on the boundary of the obstacle.
Moreover, there is no regularization term in the cost function that could penalize the variations of $\partial D$ or $(\lambda,\mu)$. Our choice enables us to capture some brutal variations of the unknown functions.
\end{remark}
\section*{Appendix}
We give below the proof of Lemma \ref{lem}.
In order to prove this lemma, we consider the local basis $(e^t_j,e^t_d)$, $j=1,d-1$, where vectors $e^t_j$ are defined by (\ref{tangent}), while $e^t_d=\eps$.
We can hence define the associated covariant basis $(f_t^i)$, $i=1,d$. Note that $f_t^i \neq e_t^i$ ($i=1,d-1$), where covariant vectors $e_t^i$ are defined by (\ref{covariant}).
We begin with the proof of the first part of Lemma \ref{lem}.
We have, denoting $\tilde{\la}=\la \circ \phi$
\[ \nabla \la_t=\sum_{i=1}^{d-1}\deri{\tilde{\la}}{\xi_i}f_t^i + \deri{\tilde{\la}}{t}f_t^d.\]
By the definition of $\tilde{\la}$, we have $\partial \tilde{\la}/\partial{t}=0$.
We hence have, with $f^i:=f_0^i$,
\be (\nabla \la_t)|_{t=0}=\sum_{i=1}^{d-1}\deri{\tilde{\la}}{\xi_i}f^i. \label{dev}\ee
It remains to compute the covariant vectors $f^i$ for $i=1,d-1$.
In this view we search $f^1$ in the form 
\[f^1=\sum_{i=1}^{d-1}\beta_i e^i+\alpha \nu.\]
The coefficients $\alpha,\beta_i$ are uniquely defined by
\[f^1\cdot e_1=1,\quad f^1\cdot e_j=0,\quad j=2,d-1,\quad f^1\cdot\eps=0.\]
This implies that
\[\beta_1=1,\quad \beta_j=0,\quad j=2,d-1,\quad \alpha(\nu\cdot\eps)=-e^1\cdot\eps.\]
As a conclusion, we have
\[f^1=e^1-\frac{1}{\nu\cdot\eps}(e^1\cdot\eps)\nu.\]
We obtain a symmetric expression for $f^i$, $i=2,d-1$ and coming back to (\ref{dev}), we obtain
\[(\nabla \la_t\cdot\nu_t)|_{t=0}=(\nabla \la_t)|_{t=0}\cdot\nu=-\sum_{i=1}^{d-1}\frac{1}{\nu\cdot\eps}(e^i\cdot\eps)\deri{\tilde{\la}}{\xi_i},\]
and lastly,
\[(\nu\cdot\eps)(\nabla \la_t\cdot\nu_t)|_{t=0}=-(\nabla_{\Gamma} \la\cdot\eps),\]
which completes the proof of the first statement of Lemma \ref{lem}.\\
Now let us give the proof of the second statement of Lemma \ref{lem}.
In this view we also need an expression of the covariant vector $f^d$.
We again search $f^d$ in the form 
\[f^d=\sum_{i=1}^{d-1}\beta_i e^i+\alpha \nu.\]
The coefficients $\alpha,\beta_i$ are now uniquely defined by
\[f^d\cdot e_i=0,\quad i=1,d-1,\quad f^d\cdot\eps=1.\]
After simple calculations, we obtain
\[f^d=\frac{1}{\nu\cdot\eps}\nu.\]
We have 
\[{\rm div} \nu_t=\sum_{i=1}^{d-1} \deri{\nu_t}{\xi_i}\cdot f_t^i + \deri{\nu_t}{t}\cdot f_t^d.\]
By differentiation of $|\nu_t|^2=1$ with respect to $\xi_i$ and $t$,
we obtain
\[\deri{\nu_t}{\xi_i}|_{t=0}.\nu=0,\quad i=1,d-1,\quad \deri{\nu_t}{t}|_{t=0}.\nu=0,\]
hence
\[ 
\left(\sum_{i=1}^{d-1} \deri{\nu_t}{\xi_i}\cdot f_t^i\right)|_{t=0}=\sum_{i=1}^{d-1} \deri{\nu}{\xi_i}\cdot e^i={\rm div}_{\Gamma} \nu,\]
and we obtain the second thesis of Lemma \ref{lem}.\\
Lastly, let us give the proof of the third statement of Lemma \ref{lem}.
Let us denote
\[G=\nabla_{\Gamma_t}u\cdot\nabla_{\Gamma_t}w\]
and $\tilde{G}_t=G \circ \phi_t$. Given the definition of surface gradient (\ref{surface}), we have
\[\nabla G= \sum_{i=1}^{d-1}\deri{\tilde{G}_t}{\xi_i}f_t^i + \deri{\tilde{G}_t}{t}f_t^d.\]
By using the expressions obtained above for the covariant vectors $f_t^i$, $i=1,d$, we obtain
\[(\nu\cdot\eps)(\nabla G\cdot\nu_t)|_{t=0}=-\sum_{i=1}^{d-1}(e^i\cdot\eps)\deri{\tilde{G}_t}{\xi_i}|_{t=0}+ \deri{\tilde{G}_t}{t}|_{t=0},\]
that is
\be (\nu\cdot\eps)(\nabla G\cdot\nu_t)|_{t=0}=-\eps_{\Gamma}\cdot\nabla_{\Gamma}(\nabla_{\Gamma}u\cdot\nabla_{\Gamma}w)+  \deri{\tilde{G}_t}{t}|_{t=0}.\label{inter}\ee
We now have to compute $\partial \tilde{G}_t/\partial t$ at $t=0$.
We have
\[\deri{\tilde{G}_t}{t}=\sum_{i,j=1}^{d-1}\deri{.}{t}\left(\deri{\tilde{u}_t}{\xi_i}e^i_t\cdot\deri{\tilde{w}_t}{\xi_j}e^j_t\right),\]
with
\[\deri{.}{t}\left(\deri{\tilde{u}_t}{\xi_i}e^i_t\cdot\deri{\tilde{w}_t}{\xi_j}e^j_t\right)=\derd{\tilde{u}_t}{\xi_i}{t}\deri{\tilde{w}_t}{\xi_j}e^i_t\cdot e^j_t + \deri{\tilde{u}_t}{\xi_i}\derd{\tilde{w}_t}{\xi_j}{t}e^i_t\cdot e^j_t + \deri{\tilde{u}_t}{\xi_i}\deri{\tilde{w}_t}{\xi_j}
\left(\deri{e^i_t}{t}\cdot e^j_t + e^i_t\cdot\deri{e^j_t}{t}\right).\]
From differentiation with respect to $t$ of 
\[(Id + t (\nabla \eps)^T)e^i_t=e^i,\]
we obtain 
\[(\nabla \eps)^T e^i_t+ (Id + t (\nabla \eps)^T))\deri{e^i_t}{t}=0,\]
hence
\[\deri{e^i_t}{t}=-(Id + t (\nabla \eps)^T)^{-1}(\nabla \eps)^T e^i_t,\]
and in particular
\[\deri{e^i_t}{t}|_{t=0}=-(\nabla \eps)^T\cdot e^i.\]
We arrive at
\begin{align*}\deri{\tilde{G}_t}{t}|_{t=0}=&\nabla_{\Gamma}\left(\deri{\tilde{u}_t}{t}|_{t=0}\right)\cdot\nabla_{\Gamma}w+\nabla_{\Gamma}u\cdot\nabla_{\Gamma}\left(\deri{\tilde{w}_t}{t}|_{t=0}\right)\\
&-\nabla_{\Gamma}u\cdot\nabla \eps\cdot\nabla_{\Gamma}w-\nabla_{\Gamma}u\cdot(\nabla \eps)^T\cdot\nabla_{\Gamma}w.\end{align*}
By using 
\[\deri{\tilde{u}_t}{t}|_{t=0}=\nabla u\cdot\eps,\quad \deri{\tilde{w}_t}{t}|_{t=0}=\nabla w\cdot\eps\]
as well as the decomposition $\eps=\eps_{\Gamma}+ (\eps\cdot\nu)\nu$, we obtain
\begin{align*}\deri{\tilde{G}_t}{t}|_{t=0}=\nabla_{\Gamma}(\nabla_{\Gamma}u\cdot\eps_{\Gamma}+&(\nabla u\cdot\nu)(\eps\cdot\nu))\cdot\nabla_{\Gamma}w+ \nabla_{\Gamma}u\cdot\nabla_{\Gamma}(\nabla_{\Gamma}w\cdot\eps_{\Gamma}+(\nabla w\cdot\nu)(\eps\cdot\nu))\\
&-\nabla_{\Gamma}u\cdot(\nabla \eps+ (\nabla \eps)^T))\cdot\nabla_{\Gamma}w.\end{align*}
We complete the proof of Lemma \ref{lem} by using equation (\ref{inter}).
\section*{Acknowledgments} 
The work of Nicolas Chaulet is supported by a grant from D\'el\'egation G\'en\'erale de l'Armement.
\bibliography{article_final_rev}

\begin{thebibliography}{10}

\bibitem{All}
{\sc G.~Allaire}, {\em Conception optimale de structures}, Springer-Verlag,
  2007.

\bibitem{haddar}
{\sc B.~Aslanyürek, H.~Haddar, and H.~Sahintürk}, {\em Generalized impedance
  boundary conditions for thin dielectric coatings with variable thickness},
  Wave Motion, 48 (2011), pp.~680 -- 699.

\bibitem{Ba09}
{\sc V.~Bacchelli}, {\em Uniqueness for the determination of unknown boundary
  and impedance with the homogeneous {R}obin condition}, Inverse Problems, 25
  (2009), pp.~015004, 4.

\bibitem{BenLem96}
{\sc A.~Bendali and K.~Lemrabet}, {\em The effect of a thin coating on the
  scattering of a time-harmonic wave for the helmholtz equation}, SIAM J. Appl.
  Math., 56 (1996), pp.~1664--1693.

\bibitem{BoChHa10report}
{\sc L.~Bourgeois, N.~Chaulet, and H.~Haddar}, {\em Identification of
  generalized impedance boundary conditions: some numerical issues}, Tech.
  Report INRIA 7449, 11 2010.

\bibitem{BoChHa11report}
\leavevmode\vrule height 2pt depth -1.6pt width 23pt, {\em On simultaneous
  identification of a scatterer and its generalized impedance boundary
  condition}, Tech. Report INRIA 7645, 6 2011.

\bibitem{BoChHa11}
\leavevmode\vrule height 2pt depth -1.6pt width 23pt, {\em Stable
  reconstruction of generalized impedance boundary conditions}, Inverse
  Problems, 27 (2011), p.~095002.

\bibitem{bourgeois_haddar}
{\sc L.~Bourgeois and H.~Haddar}, {\em Identification of generalized impedance
  boundary conditions in inverse scattering problems}, Inverse Probl. Imaging,
  4 (2010), pp.~19--38.

\bibitem{CaCo04}
{\sc F.~Cakoni and D.~Colton}, {\em The determination of the surface impedance
  of a partially coated obstacle from far field data}, SIAM J. Appl. Math., 64
  (2003/04), pp.~709--723 (electronic).

\bibitem{CaKrSc10}
{\sc F.~Cakoni, R.~Kress, and C.~Schuft}, {\em Integral equations for shape and
  impedance reconstruction in corrosion detection}, Inverse Problems, 26
  (2010), pp.~095012, 24.

\bibitem{book-CK98}
{\sc D.~Colton and R.~Kress}, {\em Inverse acoustic and electromagnetic
  scattering theory}, vol.~93 of Applied Mathematical Sciences,
  Springer-Verlag, Berlin, second~ed., 1998.

\bibitem{DurHadJol06}
{\sc M.~Durufl{\'e}, H.~Haddar, and J.~Joly}, {\em Higher order generalized
  impedance boundary conditions in electromagnetic scattering problems}, C.R.
  Physique, 7 (2006), pp.~533--542.

\bibitem{HadJol02}
{\sc H.~Haddar and P.~Joly}, {\em Stability of thin layer approximation of
  electromagnetic waves scattering by linear and nonlinear coatings}, J.
  Comput. Appl. Math., 143 (2002), pp.~201--236.

\bibitem{HadJolNgu05}
{\sc H.~Haddar, P.~Joly, and H.-M. Nguyen}, {\em Generalized impedance boundary
  conditions for scattering by strongly absorbing obstacles: the scalar case},
  Math. Models Methods Appl. Sci., 15 (2005), pp.~1273--1300.

\bibitem{haddar_kress}
{\sc H.~Haddar and R.~Kress}, {\em On the {F}r\'echet derivative for obstacle
  scattering with an impedance boundary condition}, SIAM J. Appl. Math., 65
  (2004), pp.~194--208 (electronic).

\bibitem{HeKinSin09}
{\sc L.~He, S.~Kindermann, and M.~Sini}, {\em Reconstruction of shapes and
  impedance functions using few far-field measurements}, Journal of
  Computational Physics, 228 (2009), pp.~717--730.

\bibitem{HenPie}
{\sc A.~Henrot and M.~Pierre}, {\em Variation et optimisation de formes},
  vol.~48 of Math\'ematiques \& Applications (Berlin) [Mathematics \&
  Applications], Springer, Berlin, 2005.
\newblock Une analyse g{\'e}om{\'e}trique. [A geometric analysis].

\bibitem{isakov}
{\sc V.~Isakov}, {\em On uniqueness in the inverse transmission scattering
  problem}, Comm. Partial Differential Equations, 15 (1990), pp.~1565--1587.

\bibitem{kirsch_kress}
{\sc A.~Kirsch and R.~Kress}, {\em Uniqueness in inverse obstacle scattering},
  Inverse Problems, 9 (1993), pp.~285--299.

\bibitem{KrePai}
{\sc R.~Kress and L.~P{\"a}iv{\"a}rinta}, {\em On the far field in obstacle
  scattering}, SIAM J. Appl. Math., 59 (1999).

\bibitem{kress_rundell}
{\sc R.~Kress and W.~Rundell}, {\em Inverse scattering for shape and
  impedance}, Inverse Problems, 17 (2001), p.~1075.

\bibitem{LiuNakSin07}
{\sc J.~J. Liu, G.~Nakamura, and M.~Sini}, {\em Reconstruction of the shape and
  surface impedance from acoustic scattering data for an arbitrary cylinder},
  SIAM J. Appl. Math., 67 (2007), pp.~1124--1146.

\bibitem{book-nedelec}
{\sc J.-C. N\'{e}d\'{e}lec}, {\em Acoustic and electromagnetic equations},
  Applied Mathematical Sciences, Springer-Verlag, Berlin, 2001.

\bibitem{Ru08}
{\sc W.~Rundell}, {\em Recovering an obstacle and a nonlinear conductivity from
  {C}auchy data}, Inverse Problems, 24 (2008), pp.~055015, 12.

\bibitem{Ser06}
{\sc P.~Serranho}, {\em A hybrid method for inverse scattering for shape and
  impedance}, Inverse Problems, 22 (2006), pp.~663--680.

\bibitem{Si10}
{\sc E.~Sincich}, {\em Stability for the determination of unknown boundary and
  impedance with a robin boundary condition}, SIAM J. Math. Anal., 42 (2010),
  pp.~2922--2943.

\bibitem{sokzol}
{\sc J.~Soko{\l}owski and J.-P. Zol{\'e}sio}, {\em Introduction to shape
  optimization}, vol.~16 of Springer Series in Computational Mathematics,
  Springer-Verlag, Berlin, 1992.
\newblock Shape sensitivity analysis.

\bibitem{FF++}
{\sc www.freefem.org/ff++}, {\em FreeFem++}.

\end{thebibliography}
\bibliographystyle{siam} 
\end{document}